\documentclass[11pt, twoside]{article}
\usepackage{amssymb}
\usepackage{amsfonts}
\usepackage{bbm}
\usepackage{mathrsfs}
\usepackage{amssymb,amsmath,amsfonts,amsthm}

\usepackage[colorlinks, citecolor=blue,pagebackref,hypertexnames=false]{hyperref}

\let\OLDthebibliography\thebibliography
\renewcommand\thebibliography[1]{
  \OLDthebibliography{#1}
  \setlength{\parskip}{0pt}
  \setlength{\itemsep}{0pt plus 0.3ex}
}

\allowdisplaybreaks

\begin{document}

\baselineskip=15.5pt
\renewcommand{\arraystretch}{2}
\arraycolsep=1pt

\theoremstyle{plain}
\newtheorem{theorem}{Theorem}[section]
\newtheorem{prop}[theorem]{Proposition}
\newtheorem{lemma}[theorem]{Lemma}
\newtheorem{cor}[theorem]{Corollary}
\newtheorem{example}[theorem]{Example}
\newtheorem{remark}[theorem]{Remark}
\newcommand{\ra}{\rightarrow}
\renewcommand{\theequation}
{\thesection.\arabic{equation}}

\newtheorem*{TheoremA}{Theorem A}
\newtheorem*{TheoremB}{Theorem B}
\newtheorem*{TheoremC}{Theorem C}

\theoremstyle{definition}
\newtheorem{definition}[theorem]{Definition}

\newcommand{\wtx}{\widetilde{{\bf x}}}
\newcommand{\wty}{\widetilde{{\bf y}}}
\newcommand{\wtz}{\widetilde{{\bf z}}}

\newcommand{\XX}{\mathcal {X}}
\newcommand{\ZZ}{\mathbb{Z}}
\newcommand{\G}{\mathop G\limits^{\circ}}
\newcommand{\GG}{{\mathop G\limits^{\circ}}_{\vartheta}}
\newcommand{\GGp}{{\mathop G\limits^{\circ}}_{\vartheta_1,\vartheta_2}}

\title
{\bf\Large On Fefferman--Stein type inequality on Shilov boundaries\\ and applications
}

\author{Ji Li}

\date{ }

\maketitle

\begin{center}
\begin{minipage}{14.9cm}
{\small{\bf Abstract:} In this paper, we establish the Fefferman--Stein type inequality for area integral and non-tangential maximal function on the Shilov boundary studied by Nagel and Stein \cite{NS04}. The technique here is inspired by Fefferman--Stein \cite{FS72} and Merryfield \cite{M} but we bypass the use of Fourier or group structure as these were not available on the polynomial domains of finite type. Direct applications include the maximal function characterisation of product Hardy space and the weak type endpoint estimate for product Calder\'on--Zygmund operators (such as the Cauchy--Szeg\H{o} projection) on the Shilov boundary.}
\end{minipage}
\end{center}

\medskip

\section{Introduction}
\setcounter{equation}{0}
 
In this paper, we establish the Fefferman--Stein type good-$\lambda$ inequality for area integral and non-tangential maximal function on a typical product space: the Shilov boundary of tensor product domains studied by Nagel and Stein \cite{NS04}.

Let $u(x, t)$  be a harmonic function in $\mathbb R^{n}\times(0,\infty)$.
The non-tangential maximal function $u^*(x)= \sup_{|x-y|<t}|u(y,t)|$ and  area integral $S(u)(x)^2=\int_{|x-y|<t}|\nabla u(y,t)|^2 t^{1-n}dydt$
are two fundamental tools in the theory of singular integrals and the related function spaces. Fefferman and Stein \cite{FS72} first showed that $\|u^*\|_{L^p(\mathbb R^n)}\approx\|S(u)\|_{L^p(\mathbb R^n)}$, $0<p\leq1$,  when $u(x,t)\to0$ as $t\to\infty$ (for $1<p<\infty$ this was known in \cite{St1958}).
The key objects in their proof are the following inequality (\cite[(7.2)]{FS72})
\begin{align}\label{good lambda FS}
|\{x\in\mathbb R^n: S(u)(x)>\lambda\}|\lesssim|\{x\in\mathbb R^n: u^*(x)>\lambda\}|
+{1\over \lambda^2}\int_0^\lambda s|\{x\in\mathbb R^n: u^*(x)>s\}|ds
\end{align}
and the other inequality of the same type but with $u^*$ and $S(u)$ interchanged. When $u$ is given by the Poisson integral of $f$, a different proof of the $L^p$ norm, $0<p\leq1$, of $u^*$ and $S(u)$ was given via atomic decomposition. Later, Gundy and Stein \cite{GS} established this result in the bi-disc for characterising the product Hardy space $H^p$, $0<p\leq1$, via using holomorphic function and martingales. The key step mirrors \eqref{good lambda FS}, applied to the area integral and non-tangential maximal function in the bi-disc.
It is natural to explore whether this is also true on the product spaces $\mathbb R^n\times \mathbb R^m$, noting that in the higher dimensional space, the analyticity is not available. Unlike the one-parameter setting, the equivalence $\|u^*\|_{L^p(\mathbb R^n\times \mathbb R^m)}\approx\|S(u)\|_{L^p(\mathbb R^n\times \mathbb R^m)}$ (for $0<p\leq1$) does not follow from atomic decomposition directly.
It is still not clear whether one can construct the atomic decomposition of $f$ directly when $\|u^*\|_{L^p(\mathbb R^n\times \mathbb R^m)}<\infty$ ($u$ given by the double Poisson integral of $f$). 

To overcome this, Merryfield \cite{M} provided a new proof of  \eqref{good lambda FS}, which bypassed the use of surface approximation (\cite{FS72}) or analyticity (\cite{GS}), and hence the inequality $\|S(u)\|_{L^p(\mathbb R^n\times \mathbb R^m)}\lesssim \|u^*\|_{L^p(\mathbb R^n\times \mathbb R^m)}$ holds (the reverse can be done by atomic decomposition directly). However, one key ingredient in \cite{M} is the Cauchy--Riemann equation for constructing a test function (for Littlewood--Paley estimate) from the given test function (for maximal function).  Thus, although we have studied the Hardy spaces via area function and atomic decomposition in the general product spaces of homogeneous type (see for example \cite{HLL2, CDLWY,HLPW}), the maximal function characterisation was only known in a few cases via establishing an analogy of \eqref{good lambda FS}: (1) the Muckenhoupt--Stein Bessel operator setting (\cite{DLWY}), where we exploited the Cauchy--Riemann type equation associated with the Bessel operators; (2) the multi-parameter flag setting of Nagel--Ricci--Stein \cite{NRS} (in  \cite{HLLW}), where we extended \cite{M} to the flag Euclidean setting, (3) the product stratified Lie groups \cite{CFLY} and the flag setting of Heisenberg group \cite{CCLLO}, where we used the group structure and explicit pointwise upper and lower bound of Poisson kernel.

Besides the bi-disc and product Lie groups, one fundamental model domain is the Shilov boundary $\widetilde M =M_1\times  M_2$ studied by Nagel and Stein \cite{NS04} which links to the $\bar\partial$-Neumann problem on decoupled  boundaries (to ease the burden of notational complexity we consider the tensor product of two domains).
Here
 each $M_j$ is an unbounded polynomial domain of finite type $m_j$ defined  as follows. 

Let $M$ (for simplicity we first drop the subscript $j$) be given as  
\begin{align}\label{manifold M}
  M:=\big\{ (z,w)\in\mathbb C^2:\  {\rm Im}(w)= \mathcal P(z) \big\},
\end{align}

\noindent where $ \mathcal P(z)$ is
a real, subharmonic, non-harmonic polynomial of degree $m$. We note that (\cite{NS04}) $M$ can be identified with 
$\mathbb C\times\mathbb R=\{ (z,\mathfrak t): z\in\mathbb C, \mathfrak t\in\mathbb R\}$. The basic (0,1) Levi vector field is then 
$\bar Z= {\partial\over\partial \bar z}-i {\partial \mathcal P\over\partial \bar z}{\partial\over\partial \mathfrak t}$, where $i^2=-1$, and we write $\bar Z = X_1+iX_2$. The real vector fields $\{X_1, X_2\}$ and their commutators of orders up to $m$  span the tangent space at each point.
Associated with the domain $M$ is the natural Carnot--Carath\'eodory metric $d({\bf x},{\bf y})$ on $M$ (for every ${\bf x},{\bf y}\in M$) and the measure $\mu$ of the nonisotropic ball $B({\bf x},r)=\{{\bf y}\in M: d({\bf x},{\bf y})<r\}$, which made $(M,d,\mu)$ as a space of homogeneous type in the sense of Coifman and Weiss \cite{CW}. Also denote $V_r({\bf x})=\mu(B({\bf x},r))$. Details are given in Section 2. 

The typical example that we have in mind (regarding further development and applications \cite{NS06,NRSW12,NRSW18}) is $\mathcal P(z)=\frac{1}{2k}|z|^{2k}$ for $z\in \mathbb C$ and $k$ to be positive integers. When $k=1$, $M$ is the boundary of the Siegel domain in $\mathbb C^2$ and it is CR-diffeomorphic to the first Heisenberg group, from which it inherits the group structure. However,  when $k\geq2$, $M$ does not have a group structure. Thus, we are more interested in the cases $k\geq2$. The $\bar\partial$-problem, sub-Laplacian,  Kohn-Laplacian and the related geometry, singular integrals have been intensively studied, see for example \cite{Chr,CNS,Diaz,NS01a,NS04,NRSW,Str}.  Very recent progresses on such model domains $M$ ($k\geq2$) include the explicit pointwise upper and lower bound for the Cauchy--Szeg\H{o} kernel and its applications to boundedness and compactness of commutators (\cite{CLTW}), and the construction that for each $k\geq2$, $M$ can be lifted to a Lie group $G$ (\cite{CLOW}), which  provided an explicit and optimal lifting Lie algebra comparing to the result of Rothschild and Stein \cite{RS75}. 

We now state our result in detail. Let $\widetilde M=M_1\times M_2$, where each $M_j$ is the example domain as above with $\mathcal P_j(z)=\frac{1}{2k_j}|z|^{2k_j}$, $j=1,2$. Let $d_j$ be the Carnot--Carath\'eodory metric on $M_j$ and $\mu_j$ be the corresponding measure. For the sake of simplicity, throughout the paper we denote $d\mu_j({\bf x}_j)=d{\bf x}_j$ and $|A|$ represents the measure of the set $A$.
Let $\mathcal L_j$ be the sub-Laplacian of $M_j$ and $P_{t_j}^{[j]}$ be the Poisson semigroup $ e^{-t_j\sqrt{\mathcal L_j}}$, $j=1,2$.
Consider the non-tangential maximal function 
$$
\mathcal{N}_P^{\beta}(f)({\bf x}_{1},{\bf x}_{2}):=\sup\limits_{\substack{ ({\bf y}_{1},{\bf y}_{2})\in\widetilde M,t_1>0,t_2>0,\\d_{1}({\bf x}_{1},{\bf y}_{1})<\beta t_{1},\\ d_{2}({\bf x}_{2},{\bf y}_{2})<\beta t_{2}}}|P_{t_1}^{[1]}P_{t_2} ^{[2]}(f)({\bf y}_{1},{\bf y}_{2})|
$$
with the constant $\beta>0$.
Consider also the Littlewood--Paley area function. Let
$$\nabla_{t_{1},M_{1}}:=(\partial_{t_{1}}, X_{1,1}, X_{1,2}),\quad \nabla_{t_{2},M_{2}}:=(\partial_{t_{2}}, X_{2,1}, X_{2,2}).$$
Then for any fixed $\beta\in(0,\infty)$, the Littlewood--Paley area function $S_P^{\beta}(f)({\bf x}_{1},{\bf x}_{2})$ 
is defined as
\begin{align*}
S_P^{\beta}(f)({\bf x}_{1},{\bf x}_{2})
:=\left(\iint_{\Gamma^{\beta}({\bf x}_{1},{\bf x}_{2})}|t_{1}\nabla_{t_{1},M_{1}}P_{t_1}^{[1]}\ t_{2}\nabla_{t_{2},M_{2}}P_{t_2}^{[2]}(f)({\bf y}_{1},{\bf y}_{2})|^{2}\frac{d{\bf y}_{1}d{\bf y}_{2}dt_{1}dt_{2}}{t_{1}V_{t_1}({\bf x}_1)t_{2}V_{t_2}({\bf x}_2)}\right)^{1/2},
\end{align*}
where $\Gamma^{\beta}({\bf x}_{1},{\bf x}_{2})=\Gamma_1^{\beta}({\bf x}_{1})\otimes\Gamma_2^{\beta}({\bf x}_{2})$
and $ \Gamma_j^{\beta}({\bf x}_{j})=\{ ({\bf y}_{j},t_j)\in M_j\times \mathbb R_+:  d_j( {\bf x}_{j},{\bf y}_{j} )<\beta t_j\} $, $j=1,2$.
For simplicity, we denote $S_P^{1}(f)({\bf x}_{1},{\bf x}_{2})$ by $S_P(f)({\bf x}_{1},{\bf x}_{2})$.

The main result of this paper is the following.
\begin{theorem}\label{main thm}
There exist $C>0$ and $\beta>1$ such that for all $f\in C_0^\infty(\widetilde M)$ and for all $\lambda>0$,
\begin{align}\label{good lambda}
&|\{({\bf x}_1,{\bf x}_2)\in \widetilde M:\  S_P(f)({\bf x}_1,{\bf x}_2)>\lambda\}|\\
&\leq C\big|\big\{({\bf x}_1,{\bf x}_2)\in \widetilde M:  \mathcal{N}_{P}^{\beta}(f)({\bf x}_1,{\bf x}_2)>\lambda\big\}\big|+\frac{C}{\lambda^{2}}\int_0^\lambda s|\{({\bf x}_1,{\bf x}_2)\in \widetilde M: \mathcal{N}_{P}^{\beta}(f)({\bf x}_1,{\bf x}_2)>s\}|ds.\nonumber
\end{align}
\end{theorem}

The main step here is to establish a modified version of good-$\lambda$ inequality, that is, to prove that 
$|\{({\bf x}_1,{\bf x}_2)\in A_\beta(\lambda):\  S_{P}(f)({\bf x}_1,{\bf x}_2)>\lambda\}|$ is bounded by the right-hand side of \eqref{good lambda}, where $A_\beta(\lambda)$ is a suitable subset of $E_\beta(\lambda):=\big\{({\bf x}_1,{\bf x}_2)\in \widetilde M:  \mathcal{N}_{P}^{\beta}(f)({\bf x}_1,{\bf x}_2)\leq\lambda\big\}$. The key idea that is different from \cite{M} is to inherit  the property of the double Poisson integral $u=P_{t_1}^{[1]}P_{t_2} ^{[2]}(\chi_{E_\beta(\lambda)})$ via a smooth function which captures the range of $u$. Thus, we bypass the restriction on a smooth function which is even, and has compact support. We only rely on the pointwise upper bound of the Poisson kernel as well as the conservation property. The proof will be given in Section \ref{Sec3} after some necessary preliminaries in Section \ref{Sec2}.

We further note that the above result also holds for $\widetilde M =M_1\times\cdots \times M_n$ for general $n>2$.
It suffices to repeat the proof by induction.

As applications, we point out that:  (1) Theorem \ref{main thm} passes the endpoint weak type estimate ($L\log L(\widetilde M)$ to weak $L^1(\widetilde M)$) from the strong maximal function $\mathcal M$ to the Littlewood--Paley area function $S_P^{\beta}$. Thus, through 
the approach from R. Fefferman \cite{RF} (local version) and the recent study \cite{CLLP} (global version), we obtain the $L\log L(\widetilde M)$ to weak $L^1(\widetilde M)$ for product Calder\'on--Zygmund operators on $\widetilde M$; (2) Theorem \ref{main thm} also gives rise to the characterisation of product Hardy space on $\widetilde M$, which was established via Littlewood--Paley area function and characterised equivalently by atomic decomposition \cite{HLL2,HLPW}.  Details are in Section \ref{Sec4}.

\section{Notation and preliminaries}\label{Sec2}
\setcounter{equation}{0}
In this section, we recall the basic geometry of the Shilov boundary $\widetilde M =M_1\times M_2$ \cite{NS04} with each $M_j$ given in \eqref{manifold M}, where $\mathcal P(z)=\frac{1}{2k}|z|^{2k}$, $k\geq2$. It is clear that the degree $m$ of $\mathcal P(z)$ is given by $m=2k$.

\subsection{Basic geometry of Carnot--Carath\'eodory space $M$}

We first recall the \emph{control metric on $M$} given in \cite{NS04} (see also \cite{NSW,NS01a, NS01b,Str}). Note that we write the complex (0,1) vector field $\overline{Z}=X_1+iX_{2}$, where $\{X_1, X_{2}\}$ are real vector fields on $M$. Define the metric $d$ on $M$ as follows. If ${\bf x},{\bf y}\in M$ and $\delta>0$, let $AC({\bf x},{\bf y},\delta)$ denote the set of absolutely continuous mapping $\gamma: [0,1]\rightarrow M$ such that $\gamma(0)={\bf x}$ and $\gamma(1)={\bf y}$, and such that for almost all $t\in[0,1]$ we have $\gamma'(t)=\alpha_1(t)X_1(\gamma(t))+ \alpha_{2}(t)X_{2}(\gamma(t))$ with $|\alpha_1(t)|^2+|\alpha_{2}(t)|^2<\delta^2$. Then we define
$$
  d({\bf x},{\bf y})=\inf\{\delta>0:\ AC({\bf x},{\bf y},\delta)\not=\emptyset\}.
$$

The corresponding nonisotropic ball is defined as 
$
  B({\bf x},\delta)=\{{\bf y}\in M:\ d({\bf x},{\bf y})<\delta\},
$
and let $V_\delta({\bf x})$ denotes its volume. 
From \cite{NS04} we know that there is a positive constant $C_d$ such that for every ${\bf x}\in M$, $\lambda\geq1$ and $\delta>0$,
\begin{eqnarray}\label{homogeneous dim}
V_{\lambda\delta}({\bf x})\leq C_d\lambda^{m} V_\delta({\bf x}).
\end{eqnarray}
We also set $ V({\bf x},{\bf y})=V_{d({\bf x},{\bf y})}({\bf x}).$ From the doubling property we observe that $ V({\bf x},{\bf y})\approx  V({\bf y},{\bf x})$ where the implicit constants are independent of ${\bf x}$ and ${\bf y}$.

\subsection{Sub-Laplacian on $M$}

Consider the sub-Laplacian $\mathcal {L}$ on $M$ in self-adjoint
form, given by
$\mathcal {L}=\sum_{j=1}^2 {X}_j^{*}{X}_j.$
Here $({X}_j^{*}\varphi, \psi)=(\varphi, {X}_j\psi)$,
where $(\varphi,\psi)=\int_{M} \varphi({\bf x})\bar{\psi}({\bf x})d{\bf x}
$, and $\varphi, \psi\in C_0^{\infty}(M)$, the space of $C^{\infty}$
functions on $M$ with compact support. 
In general, we have
${X}_j^{*}= -{X}_j+a_j$, where $a_j\in C^\infty(M)$. In our particular setting, we see that $a_j=0$. That is
\begin{eqnarray*}
\mathcal {L}=-\sum_{j=1}^2 {X}_j^2.
\end{eqnarray*}

The solution of the following initial value problem for the heat
equation,
\begin{eqnarray*}
{{\partial u} \over {\partial s}}({\bf x},s)+      \mathcal {L}_{\bf x} u({\bf x},s)=0
\end{eqnarray*}
with $u({\bf x},0)=f({\bf x})$, is given by $u({\bf x},s)=H_s(f)({\bf x})$, where $H_s$ is
the operator given via the spectral theorem by $H_s=e^{-s\mathcal
{L}}$, and an appropriate self-adjoint extension of the non-negative
operator $\mathcal {L}$ initially defined on $C_0^{\infty}(M)$. For $f\in L^2(M)$,
\begin{eqnarray*}
H_s(f)({\bf x})=\int_M H(s,{\bf x},{\bf y})f({\bf y})d{\bf y}.
\end{eqnarray*}
Moreover, $H(s,{\bf x},{\bf y})$ has some nice properties (see Proposition 2.3.1
in \cite{NS04} and Theorem 2.3.1 in \cite{NS01a}). We restate them
as follows:
\begin{itemize}
\item[(1)] $H(s,{\bf x},{\bf y})\in C^{\infty}\big([0,\infty)\times M\times M
\backslash \{s=0\ {\rm and}\ {\bf x}={\bf y} \}\big).$
\item[(2)] For every integer $N\geq 0$,
\begin{align*}
|\partial_s^j\partial_X^L\partial_Y^K H(s,{\bf x},{\bf y})|
 \lesssim \frac{\displaystyle 1 }{\displaystyle
(d({\bf x},{\bf y})+\sqrt{s})^{2j+K+L} } \frac{\displaystyle 1 }{\displaystyle
V({\bf x},{\bf y})+V_{\sqrt{s}}({\bf x}) +V_{\sqrt{s}}({\bf y}) } \bigg(\frac{\displaystyle
\sqrt{s} }{\displaystyle d({\bf x},{\bf y})+\sqrt{s} } \bigg)^{N},
\end{align*}
\item[(3)] For each integer $L\geq 0$ there exist an integer $N_L$ and a
constant $C_L$ so that if $\varphi\in C_0^{\infty}(B({\bf x}_0,\delta))$,
then for all $s\in(0,\infty),$
$$ |\partial_X^L H_s[\varphi]({\bf x}_0)|\leq C_L\delta^{-L}\sup_{\bf x}\sum_{|J|\leq N_L}\delta^{|J|}|\partial_X^J\varphi({\bf x})|. $$
\item[(4)] For all $(s,{\bf x},{\bf y})\in (0,\infty)\times M \times M$,
$H(s,{\bf x},{\bf y})=H(s,{\bf y},{\bf x}),$
$H(s,{\bf x},{\bf y})\geq 0$.
\item[(5)] For all $(s,{\bf x})\in (0,\infty)\times M$, $\int_M H(s,{\bf x},{\bf y})d{\bf y}=1.$
\item[(6)]  For $1\leq p\leq \infty$, $\|H_s[f]\|_{L^p(M)}\leq
\|f\|_{L^p(M)}$.
\item[(7)]  For every $\varphi\in C_0^{\infty}(M)$ and  $1\leq p<\infty$,
$\lim\limits_{s\rightarrow 0}\|H_s[\varphi]-\varphi\|_{L^p(M)}=0$.
\end{itemize}

\subsection{Poisson kernel estimate on $M$}
From (2) in Section 2.2, we see that there is a positive constant $C_H$ such that for all $s>0$, ${\bf x},{\bf y}\in M$,
\begin{align*}
|H(s,{\bf x},{\bf y})|
 \leq C_H  \frac{\displaystyle 1 }{\displaystyle
V_{\sqrt{s}}({\bf x}, {\bf y})+V_{\sqrt{s}}({\bf x}) +V_{\sqrt{s}}({\bf y}) } \bigg(\frac{\displaystyle
\sqrt{s} }{\displaystyle d({\bf x},{\bf y})+\sqrt{s} } \bigg)^{N}.
\end{align*}

Let $P_{t}({\bf x},{\bf y})$ be the kernel of the Poisson semigroup $\mathrm{e}^{-t\sqrt{\mathcal{L}}}$.
The estimates for $P_{t}({\bf x},{\bf y})$ follow from the subordination formula
\begin{equation*}
\mathrm{e}^{-t\sqrt{\mathcal{L}}}
=\frac{1}{2\sqrt{\pi}}\int_0^\infty
\frac{t\mathrm{e}^{-{t^2}\over{4s}}}{\sqrt{s}}\mathrm{e}^{-s \mathcal{L}}\frac{ds}{s},
\end{equation*}
the doubling property of the measure and the estimate of $|H(s,{\bf x},{\bf y})|$ as above. We have
that 
for each ${\bf x}\in M$ and $t>0$
\begin{equation}\label{int P =1}
    \int_M P_t({\bf x},{\bf y})d{\bf y}=1,
\end{equation}
and
that there exists $C_P>0$ such that  for each ${\bf x},{\bf y}\in M$ and $t>0$, 
\begin{equation}\label{P size}
    |P_t({\bf x},{\bf y})|\leq C_P \frac{1}{V({\bf x},{\bf y})+V_t({\bf x})+V_t({\bf y})}\cdot\frac{t}{t+d({\bf x},{\bf y})}.
\end{equation}
Here $C_P$ depends on $C_H$ and the upper dimension $m$ as in \eqref{homogeneous dim}.  From this size estimate we see that for every $f\in L^2(M)$,
\begin{equation}\label{P f M f}
    |P_t(f)({\bf x})|\leq C_P(C_d2^m+1) \mathcal M(f)({\bf x}),
\end{equation}
where $C_d$ and $m$ are the constants from \eqref{homogeneous dim} and 
$\mathcal M$ is the Hardy--Littlewood maximal operator such that 
$$\mathcal M(f)({\bf x}) =\sup_{B\ni x}{1\over |B|}\int_B |f({\bf y})| d{\bf y},$$
where $B$ runs over all metric balls in $M$.

\subsection{Basic geometry of Shilov boundary $\widetilde M=M_1\times M_2$}

Consider $\widetilde M=M_1\times M_2$ such that $M_j= \big\{(z_j,w_j)\in\mathbb{C}^{2}:\ {\rm Im} (w_{j})= \mathcal  P_j(z_j)
\big\}$ with the vector fields $X_{j,1}$ and $X_{j,2}$, $j=1,2$. We denote $\vec{\bf x}= ({\bf x}_1,{\bf x}_2)\in M_1\times M_2$.

The nonisotropic distance $d_j$ on $M_j$ can be regarded as a function on $M_j$ which depends only on the variables $(z_j,t_j)$, where $t_j={\rm Re} (w_{j})$. In addition, there is a nonisotropic metric $d_{\sum}$ on $\widetilde{M}$ induced by all real vector fields $\{X_{1,1},X_{1,2},X_{2,1},X_{2,2}\}$. If $\vec{\bf x},\vec{\bf y}\in \widetilde M$ and $\delta>0$, let $AC(\vec{\bf x},\vec{\bf y},\delta)$ denote the set of absolutely continuous mappings $\gamma:[0,1]\rightarrow \widetilde{M}$ such that $\gamma(0)=\vec{\bf x}$ and $\gamma(1)=\vec{\bf y}$, and such that for almost every $t\in [0,1]$ we have $\gamma'(t)=\sum_{j=1}^{2}(\alpha_{j,1}(t)X_{j,1}(\gamma(t))+\alpha_{j,2}(t)X_{j,2}(\gamma(t)))$ with $\sum_{j=1}^{2}(|\alpha_{j,1}(t)|^2+|\alpha_{j,2}(t)|^2)<\delta^2$. Then
$$
    d_{\sum}(\vec{\bf x},\vec{\bf y})=\inf\{\delta>0 \mid AC(\vec{\bf x},\vec{\bf y},\delta)\not=\emptyset\}.
$$

Similar to \eqref{P f M f}, we also have that for every $f\in L^2(\widetilde M)$,
\begin{equation}\label{P f M f prod}
    |P_{t_1}^{[1]}P_{t_2} ^{[2]}(f)({\bf x}_{1},{\bf x}_{2})|\leq C_0 \mathcal M_S(f)({\bf x}_1,{\bf x}_2),
\end{equation}
where $C_0$ depends on the constant in  \eqref{P f M f}, $P_{t_j}^{[j]}$ denotes the Poisson semigroup on $M_j$, and 
$\mathcal M_S$ is the strong maximal function on $\widetilde M$ such that 
$$\mathcal M_S(f)({\bf x}_1,{\bf x}_2) =\sup_{B_1\times B_2\ni ({\bf x}_1,{\bf x}_2)}{1\over |B_1\times B_2|}\int_{B_1\times B_2} |f({\bf y}_1,{\bf y}_2)| d{\bf y}_1d{\bf y}_2,$$
where $B_j$ runs over all metric balls in $M_j$ for $j=1,2$.

\section{Proof of Theorem \ref{main thm}: Fefferman--Stein type inequality on~$\widetilde M$}\label{Sec3}
\setcounter{equation}{0}

For $j=1,2$, let $M_j$ be the model domain as defined in Section \ref{Sec2}, with the vector fields $X_{j,1}$ and $X_{j,2}$ and the sub-Laplacian 
$\mathcal L_j$.

We aim to prove that 
there exist constants $C>0$ and $\beta>1$ such that for $f\in C_0^\infty(\widetilde M)$ and for all $\alpha>0$,
\begin{align}\label{good lambda group}
&|\{({\bf x}_1,{\bf x}_2)\in \widetilde M:\  S_{P}(f)({\bf x}_1,{\bf x}_2)>\alpha\}|\\
&\leq C\big|\big\{({\bf x}_1,{\bf x}_2)\in \widetilde M:\  \mathcal{N}_{P}^{\beta}(f)({\bf x}_1,{\bf x}_2)>\alpha\big\}\big|+\frac{C}{\alpha^{2}}\iint_{E_\beta(\alpha)}\mathcal{N}_{P}^{\beta}(f)({\bf x}_1,{\bf x}_2)^{2}d{\bf x}_{1}d{\bf x}_{2},\nonumber
\end{align}
where 
$$E_\beta(\alpha)=\big\{({\bf x}_1,{\bf x}_2)\in \widetilde M:\  \mathcal{N}_{P}^{\beta}(f)({\bf x}_1,{\bf x}_2)\leq\alpha\big\}.$$
We denote 
$$E_\beta(\alpha)^c= \widetilde M\backslash E_\beta(\alpha)=\big\{({\bf x}_1,{\bf x}_2)\in \widetilde M:\  \mathcal{N}_{P}^{\beta}(f)({\bf x}_1,{\bf x}_2)>\alpha\big\}.$$

For $j=1,2$, we note that for any $f\in L^{2}(M_j)$,  $u_{j}( {\bf x}_j, t_j):=e^{- t_j\sqrt{\mathcal L_j}}(f)( {\bf x}_j)$, $( {\bf x}_j, t_j)\in M_j\times \mathbb{R}_{+}$, is harmonic, in the sense that
\begin{align*}
\Delta_{ t_j, M_j}u_{j}( {\bf x}_j, t_j)=\mathcal L_ju_{j}( {\bf x}_j, t_j)-\partial_{ t_j}^{2} u_{j}( {\bf x}_j, t_j)=0,
\end{align*}
where $\Delta_{ t_j, M_j}:=\mathcal L_j-\partial_{ t_j}^{2}$ and we use the fact that 
$\partial_{ t_j}^{2} u_{j}( {\bf x}_j, t_j) = \partial_{ t_j}^{2} e^{- t_j\sqrt{\mathcal L_j}}(f)( {\bf x}_j) = \mathcal L_ju_{j}( {\bf x}_j, t_j)$. 
Consequently, for $j=1,2$,
for the gradient $\nabla_{ t_j, M_j}=(\partial_{ t_j},\nabla_{ M_j})=(\partial_{ t_j},X_{j,1},X_{j,2})$,
the following formula holds: for every $( {\bf x}_j, t_j)\in M_j\times \mathbb{R}_{+}$,
\begin{align}\label{harmonic}
2|\nabla_{ t_j, M_j}u_{j}( {\bf x}_j, t_j)|^{2}&=2|\nabla_{ t_j, M_j}u_{j}( {\bf x}_j, t_j)|^{2}-2u_{j}( {\bf x}_j, t_j) \Delta_{ t_j, M_j}u_{j}( {\bf x}_j, t_j)\\
&=-\Delta_{ t_j, M_j}\left(u_{j}^{2}( {\bf x}_j, t_j)\right).\nonumber
\end{align}

Next, for all $\alpha>0$ and $f\in L^{1}(\widetilde M)$ satisfying $\mathcal{N}_{P}^{\beta}(f)\in L^{1}(\widetilde M)$, define
\begin{align*}
A_{\beta}(\alpha):=\left\{( {\bf x}_{1},{\bf x}_{2})\in \widetilde M: \mathcal{M}_{S}(\chi_{E_\beta(\alpha)^c})( {\bf x}_{1},{\bf x}_{2})\leq\frac{1}{10 C_0}\right\},
\end{align*}
where $C_0$ is the constant in \eqref{P f M f prod}.

From the definition, it is direct to see that 
\begin{align*}
E_\beta(\alpha)^c\subset A_{\beta}(\alpha)^{c} =\widetilde M\backslash A_{\beta}(\alpha)
\end{align*}
and hence $A_{\beta}(\alpha)\subset E_\beta(\alpha)$.

Next,
from the  $L^{2}$-boundedness of the strong maximal function $\mathcal{M}_{S}$, we see that
\begin{align}\label{Abetac leq Ebetac}
|A_{\beta}(\alpha)^{c}|\leq C\big\|\mathcal{M}_{S}(\chi_{ E_\beta(\alpha)^c})\big\|_{L^2(\widetilde M)}^2\leq C|E_\beta(\alpha)^c|,
\end{align}
where the constant $C$ is independent of $\alpha$ and $\beta$.
Then we split
\begin{align}\label{good lambda}
&|\{({\bf x}_1,{\bf x}_2)\in \widetilde M:\  S_{P}(f)({\bf x}_1,{\bf x}_2)>\alpha\}|\\
&\leq |\{({\bf x}_1,{\bf x}_2)\in A_\beta(\alpha)^c:\  S_{P}(f)({\bf x}_1,{\bf x}_2)>\alpha\}|
+|\{({\bf x}_1,{\bf x}_2)\in A_\beta(\alpha):\  S_{P}(f)({\bf x}_1,{\bf x}_2)>\alpha\}|\nonumber\\
&\leq C\big|E_\beta(\alpha)^c\big|+\frac{1}{\alpha^{2}}\iint_{A_\beta(\alpha)} S_P(f)({\bf x}_1,{\bf x}_2)^{2}d{\bf x}_{1}d{\bf x}_{2},\nonumber
\end{align}
where the last inequality follows from \eqref{Abetac leq Ebetac} and from Chebyshev's inequality.  Now it suffices to estimate the second term in the right-hand side of last 
inequality above.

Let $$g( {\bf x}_{1},{\bf x}_{2}):=\chi_{E_\beta(\alpha)}({\bf x}_{1},{\bf x}_{2})\quad{\rm and}\quad W_{\beta}:=\bigcup_{({\bf x}_{1},{\bf x}_{2})\in A_{\beta}(\alpha)}\Gamma({\bf x}_{1},{\bf x}_{2}).$$

We first note that for $({\bf y}_{1},{\bf y}_{2})\in\widetilde M$ and $t_1,t_2>0$,
\begin{align*}
P_{t_1}^{[1]}P_{t_2}^{[2]}(g)( {\bf y}_{1},{\bf y}_{2})  
&=P_{t_1}^{[1]}P_{t_2}^{[2]}\big(1- \chi_{E_\beta(\alpha)^c} \big)( {\bf y}_{1},{\bf y}_{2})  \\
&=1- P_{t_1}^{[1]}P_{t_2}^{[2]}\big( \chi_{E_\beta(\alpha)^c} \big)( {\bf y}_{1},{\bf y}_{2}),
\end{align*}
where we use the fact that $P_{t_1}^{[1]}P_{t_2}^{[2]}(1)=1.$
Next, note that for $({\bf y}_{1},{\bf y}_{2},t_1,t_2)\in W_{\beta}$, there is $({\bf x}_{1},{\bf x}_{2})\in A_{\beta}(\alpha)$ such that 
$({\bf y}_{1},{\bf y}_{2},t_1,t_2)\in \Gamma({\bf x}_{1},{\bf x}_{2})$, and hence
$$P_{t_1}^{[1]}P_{t_2}^{[2]}\big( \chi_{E_\beta(\alpha)^c} \big)( {\bf y}_{1},{\bf y}_{2}) \leq C_0\mathcal{M}_{S}(\chi_{E_\beta(\alpha)^c })( {\bf x}_{1},{\bf x}_{2})< {1\over 10},$$
where the first inequality follows from \eqref{P f M f prod} and the last inequality follows from the fact that $({\bf x}_{1},{\bf x}_{2})\in A_{\beta}(\alpha)$.

Then we see that for every $({\bf y}_{1},{\bf y}_{2},t_1,t_2)\in W_{\beta}$, we obtain that 
\begin{align*}
P_{t_1}^{[1]}P_{t_2}^{[2]}(g)( {\bf y}_{1},{\bf y}_{2})  > {9\over 10}.
\end{align*}

Next, we claim that if $\beta$ is chosen sufficient large, then there is a constant $C_{1}\in(0,{9\over10})$, such that for any $({\bf y}_{1},{\bf y}_{2},t_{1},t_{2})\in \big(\widetilde{W}_{\beta}\big)^c:=\Big(\widetilde M\times[0,\infty)\times[0,\infty)\Big)\Big\backslash \widetilde{W}_{\beta}$,
\begin{align}\label{P g small}
P_{t_1}^{[1]}P_{t_2}^{[2]}(g)( {\bf y}_{1},{\bf y}_{2}) \leq C_1,
\end{align}
where 
$$\widetilde{W}_{\beta}:=\bigcup_{({\bf x}_1,{\bf x}_2)\in E_\beta(\alpha)}\Gamma^{\beta}({\bf x}_1,{\bf x}_2).$$
In fact, for every 
$({\bf y}_{1},{\bf y}_{2},t_{1},t_{2}) \in \big(\widetilde{W}_{\beta}\big)^c$, we see that for any $({\bf z}_{1},{\bf z}_{2})\in E_\beta(\alpha)$, we have either $d_{1}({\bf y}_1,{\bf z}_1)\geq \beta  t_1$, or $d_{2}({\bf y}_2,{\bf z}_2)\geq \beta  t_2$, or both. Hence,
 we have
\begin{align*}
0\leq P_{t_1}^{[1]}P_{t_2}^{[2]}(g)( {\bf y}_{1},{\bf y}_{2})
&=\iint_{\widetilde M}\chi_{E_\beta(\alpha)}({\bf z}_{1},{\bf z}_{2})P_{t_{1}}^{[1]}({\bf y}_{1},{\bf z}_{1})P_{t_{2}}^{[2]}({\bf y}_{2}, {\bf z}_{2})d{\bf z}_{1}d{\bf z}_{2}\\
&\leq \int_{d_{1}({\bf y}_1,{\bf z}_1)\geq \beta  t_1}P_{t_{1}}^{[1]}({\bf y}_{1},{\bf z}_{1})d{\bf z}_{1} \int_{M_2}P_{t_{2}}^{[2]}({\bf y}_{2}, {\bf z}_{2})d{\bf z}_{2}\\
&\qquad+ \int_{M_1}P_{t_{1}}^{[1]}({\bf y}_{1},{\bf z}_{1})d{\bf z}_{1} \int_{d_{2}({\bf y}_2,{\bf z}_2)\geq \beta  t_2}P_{t_{2}}^{[2]}({\bf y}_{2}, {\bf z}_{2})d{\bf z}_{2}\\
&\leq \int_{d_{1}({\bf y}_1,{\bf z}_1)\geq \beta  t_1}P_{t_{1}}^{[1]}({\bf y}_{1},{\bf z}_{1})d{\bf z}_{1} 
+ \int_{d_{2}({\bf y}_2,{\bf z}_2)\geq \beta  t_2}P_{t_{2}}^{[2]}({\bf y}_{2}, {\bf z}_{2})d{\bf z}_{2},
\end{align*}
where the last inequality follows from the conservation property \eqref{int P =1}.
To continue, by decomposing  $\{z_j\in M_j: d_{j}({\bf y}_j,{\bf z}_j)\geq \beta  t_j\}$, $j=1,2$,  into annuli and using the size estimate \eqref{P size},
we have
\begin{align*}
0\leq P_{t_1}^{[1]}P_{t_2}^{[2]}(g)( {\bf y}_{1},{\bf y}_{2})
\leq {C\over \beta}
\rightarrow 0,\ ({\rm as}\ \beta\rightarrow \infty),
\end{align*}
  where the constant $C$ depends on the constant $C_d$ and $m$ in \eqref{homogeneous dim} and on $C_P$ in \eqref{P size}.
Thus, there is some $\beta>1$ such that 
 our claim \eqref{P g small} holds.
 
 This also shows that if $P_{t_1}^{[1]}P_{t_2}^{[2]}(g)( {\bf y}_{1},{\bf y}_{2}) > C_1,$ then $({\bf y}_{1},{\bf y}_{2},t_{1},t_{2}) \in \widetilde{W}_{\beta}$.

To continue, we now choose a smooth cut-off function $\varphi(t)\in C^{\infty}(\mathbb{R})$ such that $\varphi(t)=1$ when $t\geq {9\over10}$ and $\varphi(t)=0$, when $t\leq C_{1}$. 

Besides, for simplicity, we  denote $v_{{\bf y}_{2},t_{2}}({\bf y}_{1}):=\nabla_{t_{2},M_{2}}P_{t_2}^{[2]}(f)( {\bf y}_{1},{\bf y}_{2})$. Then,
\begin{align}\label{a1}
&\iint_{A_{\beta}(\alpha)}S_{P}(f)( {\bf x}_{1},{\bf x}_{2})^2d{\bf x}_{1}d{\bf x}_{2}\nonumber\\
&=\iint_{A_{\beta}(\alpha)}\iint_{\Gamma({\bf x}_{1},{\bf x}_{2})}|t_{1}\nabla_{t_{1},M_{1}}P_{t_1}^{[1]}\ t_{2}\nabla_{t_{2},M_{2}}P_{t_2}^{[2]}(f)({\bf y}_{1},{\bf y}_{2})|^{2} \frac{d{\bf y}_{1}d{\bf y}_{2}dt_{1}dt_{2}}{t_{1}V_{t_1}({\bf x}_1)t_{2}V_{t_2}({\bf x}_2)}d{\bf x}_{1}d{\bf x}_{2}\nonumber\\
&=\iint_{A_{\beta}(\alpha)}\iint_{\Gamma({\bf x}_{1},{\bf x}_{2})} \left|t_{1}\nabla_{t_{1},M_{1}}P_{t_1}^{[1]}t_{2}v_{{\bf y}_{2},t_{2}}({\bf y}_{1})\right|^{2} \frac{d{\bf y}_{1}d{\bf y}_{2}dt_{1}dt_{2}}{t_{1}V_{t_1}({\bf x}_1)t_{2}V_{t_2}({\bf x}_2)}d{\bf x}_{1}d{\bf x}_{2}\nonumber\\
&\leq\iiiint_{W_{\beta}} \left|\nabla_{t_{1},M_{1}}P_{t_1}^{[1]}v_{{\bf y}_{2},t_{2}}({\bf y}_{1})\right|^{2}t_{1}t_{2}dt_{1}dt_{2}d{\bf y}_{1}d{\bf y}_{2}\nonumber\\
&\leq\iiiint_{\widetilde M\times \mathbb{R}_{+}\times\mathbb{R}_{+}} \left|\nabla_{t_{1},M_{1}} P_{t_1}^{[1]}v_{{\bf y}_{2},t_{2}}({\bf y}_{1})\right|^{2}\nonumber\\
&\hspace{4.3cm} \times \left|\varphi\Big(P_{t_1}^{[1]}P_{t_2}^{[2]}(g)( {\bf y}_{1},{\bf y}_{2}) \Big)\right|^{2}t_{1}t_{2}dt_{1}dt_{2}d{\bf y}_{1}d{\bf y}_{2}.
\end{align}
To continue, we note that  $P_{t_1}^{[1]}v_{{\bf y}_{2},t_{2}}({\bf y}_{1})$ as a function of $({\bf y}_{1},t_{1})$ is harmonic in $M_1\times\mathbb R_+$. Moreover, $P_{t_1}^{[1]}P_{t_2}^{[2]}(g)( {\bf y}_{1},{\bf y}_{2}) $ as a function of $({\bf y}_{1},t_{1})$ is also harmonic in $M_1\times\mathbb R_+$.
 Hence, by using \eqref{harmonic} the following equality holds:
\begin{align}\label{domination}
&\left|\nabla_{t_{1},M_{1}} P_{t_1}^{[1]}v_{{\bf y}_{2},t_{2}}({\bf y}_{1})\right|^{2} \left| \varphi\Big(P_{t_1}^{[1]}P_{t_2}^{[2]}(g)( {\bf y}_{1},{\bf y}_{2}) \Big)\right|^{2}\\[5pt]
&= -\frac{1}{2}\Delta_{t_{1},M_{1}}\left(\Big(P^{[1]}_{t_{1}} v_{{\bf y}_{2},t_{2}}({\bf y}_{1})\Big)^{2}\cdot \varphi\Big(P_{t_1}^{[1]}P_{t_2}^{[2]}(g)( {\bf y}_{1},{\bf y}_{2}) \Big)^{2}\right)\nonumber\\[5pt]
&\qquad-4P^{[1]}_{t_{1}} v_{{\bf y}_{2},t_{2}}({\bf y}_{1}) \nabla_{t_{1},M_{1}} P^{[1]}_{t_{1}} v_{{\bf y}_{2},t_{2}}({\bf y}_{1})\nonumber \\[5pt]
&\hspace{1.5cm} \times \varphi\Big(P_{t_1}^{[1]}P_{t_2}^{[2]}(g)( {\bf y}_{1},{\bf y}_{2}) \Big)\varphi^{\prime}\Big(P_{t_1}^{[1]}P_{t_2}^{[2]}(g)( {\bf y}_{1},{\bf y}_{2}) \Big) \nabla_{t_{1},M_{1}} P_{t_1}^{[1]}P_{t_2}^{[2]}(g)( {\bf y}_{1},{\bf y}_{2})\nonumber\\[5pt]
&\qquad-|P^{[1]}_{t_{1}} v_{{\bf y}_{2},t_{2}}({\bf y}_{1})|^{2}\varphi^{\prime}\Big( P_{t_1}^{[1]}P_{t_2}^{[2]}(g)( {\bf y}_{1},{\bf y}_{2}) \Big)^{2}\left|\nabla_{t_{1},M_{1}} P_{t_1}^{[1]}P_{t_2}^{[2]}(g)( {\bf y}_{1},{\bf y}_{2})\right|^{2}\nonumber\\[5pt]
&\qquad-|P^{[1]}_{t_{1}} v_{{\bf y}_{2},t_{2}}({\bf y}_{1})|^{2}\varphi\Big( P_{t_1}^{[1]}P_{t_2}^{[2]}(g)( {\bf y}_{1},{\bf y}_{2})\Big)\varphi^{\prime\prime}\Big(P_{t_1}^{[1]}P_{t_2}^{[2]}(g)( {\bf y}_{1},{\bf y}_{2})\Big)\left|\nabla_{t_{1},M_{1}} P_{t_1}^{[1]}P_{t_2}^{[2]}(g)( {\bf y}_{1},{\bf y}_{2})\right|^{2}\nonumber\\[5pt]
&=:f_{1}( {\bf y}_{1},{\bf y}_{2},t_{1},t_{2})+f_{2}( {\bf y}_{1},{\bf y}_{2},t_{1},t_{2})+f_{3}( {\bf y}_{1},{\bf y}_{2},t_{1},t_{2})+f_{4}( {\bf y}_{1},{\bf y}_{2},t_{1},t_{2}).\nonumber
\end{align}

We note that by Young's inequality, 
\begin{align*}
|f_{2}( {\bf y}_{1},{\bf y}_{2},t_{1},t_{2})|&\leq {1\over 10}\left|\nabla_{t_{1},M_{1}} P_{t_1}^{[1]}v_{{\bf y}_{2},t_{2}}({\bf y}_{1})\right|^{2} \left| \varphi\Big(P_{t_1}^{[1]}P_{t_2}^{[2]}(g)( {\bf y}_{1},{\bf y}_{2}) \Big)\right|^{2}\\
&\quad+40 |P^{[1]}_{t_{1}} v_{{\bf y}_{2},t_{2}}({\bf y}_{1})|^2 \left|\varphi^{\prime}\Big(P_{t_1}^{[1]}P_{t_2}^{[2]}(g)( {\bf y}_{1},{\bf y}_{2}) \Big)\right|^2 |\nabla_{t_{1},M_{1}} P_{t_1}^{[1]}P_{t_2}^{[2]}(g)( {\bf y}_{1},{\bf y}_{2})|^2\\
&=:f_{21}( {\bf y}_{1},{\bf y}_{2},t_{1},t_{2})+f_{22}( {\bf y}_{1},{\bf y}_{2},t_{1},t_{2}). 
\end{align*}

We can see that the integral 
$$\iiiint_{\widetilde M\times \mathbb{R}_{+}\times\mathbb{R}_{+}}  f_{21}( {\bf y}_{1},{\bf y}_{2},t_{1},t_{2}) t_{1}t_{2}dt_{1}dt_{2}d{\bf y}_{1}d{\bf y}_{2}.$$
 can be absorbed by the right-hand side of \eqref{a1}, while $f_{22}( {\bf y}_{1},{\bf y}_{2},t_{1},t_{2})$ is quite similar to $f_{3}( {\bf y}_{1},{\bf y}_{2},t_{1},t_{2})$ and $f_{4}( {\bf y}_{1},{\bf y}_{2},t_{1},t_{2})$.
Hence, we further have
\begin{align}\label{domination 2}
&f_{22}( {\bf y}_{1},{\bf y}_{2},t_{1},t_{2})+f_{3}( {\bf y}_{1},{\bf y}_{2},t_{1},t_{2})+f_{4}( {\bf y}_{1},{\bf y}_{2},t_{1},t_{2})\\
&\leq 40 |P^{[1]}_{t_{1}} v_{{\bf y}_{2},t_{2}}({\bf y}_{1})|^2 \left|\Phi'\Big(P_{t_1}^{[1]}P_{t_2}^{[2]}(g)( {\bf y}_{1},{\bf y}_{2}) \Big)\right|^2 \left|\nabla_{t_{1},M_{1}} P_{t_1}^{[1]}P_{t_2}^{[2]}(g)( {\bf y}_{1},{\bf y}_{2})\right|^2 \nonumber\\
&=: \mathfrak f_{2}( {\bf y}_{1},{\bf y}_{2},t_{1},t_{2}),\nonumber
\end{align}
where we choose $\Phi(t)$ such that 
\begin{align*}
\Phi'(t)=((\varphi^{\prime}(t))^{4}+(\varphi(t)\varphi^{\prime\prime}(t))^{2})^{1/4},\qquad {\rm and }\qquad \Phi(C_1)=0.
\end{align*}
Note that $\Phi'(t)\geq 0$ and $\Phi'(t)=0$ for $t\leq C_1$ or $t> \frac{9}{10}$. In addition, via assuming  $\Phi(C_1)=0$, we see that   $\Phi$ exhibits behavior similar to $\phi$. 

To continue,
note that the right-hand side of \eqref{a1} is bounded by 
\begin{align*}
&{10\over9}\left|\iiiint_{\widetilde M\times \mathbb{R}_{+}\times\mathbb{R}_{+}} f_{1}( {\bf y}_{1},{\bf y}_{2},t_{1},t_{2})t_{1}t_{2}dt_{1}dt_{2}d{\bf y}_{1}d{\bf y}_{2}\right|\\
&\qquad+{10\over9}\left|\iiiint_{\widetilde M\times \mathbb{R}_{+}\times\mathbb{R}_{+}}\mathfrak f_{2}( {\bf y}_{1},{\bf y}_{2},t_{1},t_{2})t_{1}t_{2}dt_{1}dt_{2}d{\bf y}_{1}d{\bf y}_{2}\right|\\
&=:{\rm I}_{1}+{\rm I}_{2}.
\end{align*}

For the term ${\rm I_{1}}$, integration by parts  yields that
\begin{align*}
&\left|\iint_{ M_1\times \mathbb{R}_{+}}f_{1}( {\bf y}_{1},{\bf y}_{2},t_{1},t_{2})t_{1}dt_{1}d{\bf y}_{1}\right|\\
&={1\over2}\left|\iint_{M_{1}\times\mathbb{R}_{+}} \mathcal L_1\left(P^{[1]}_{t_{1}}v_{{\bf y}_{2},t_{2}}({\bf y}_{1})^{2}\cdot \varphi(P_{t_1}^{[1]}P_{t_2}^{[2]}(g)( {\bf y}_{1},{\bf y}_{2}) )^{2}\right)t_{1}dt_{1}d{\bf y}_{1}\right.\\
&\qquad\left.-\iint_{M_{1}\times\mathbb{R}_{+}} \partial_{t_1}^2\left(P^{[1]}_{t_{1}}v_{{\bf y}_{2},t_{2}}({\bf y}_{1})^{2}\cdot \varphi(P_{t_1}^{[1]}P_{t_2}^{[2]}(g)( {\bf y}_{1},{\bf y}_{2}) )^{2}\right)t_{1}dt_{1}d{\bf y}_{1}\right|\\
&={1\over2}\left|\int_{\mathbb{R}_{+}} t_{1}\nabla_{ M_1}\left(P^{[1]}_{t_{1}}v_{{\bf y}_{2},t_{2}}({\bf y}_{1})^{2}\cdot \varphi(P_{t_1}^{[1]}P_{t_2}^{[2]}(g)( {\bf y}_{1},{\bf y}_{2}) )^{2}\right)\bigg|_{d_1({\bf 0},{\bf y}_1)=\infty}dt_{1}\right.\\
&\qquad\left.-\int_{M_{1}} \partial_{t_1}\left(P^{[1]}_{t_{1}}v_{{\bf y}_{2},t_{2}}({\bf y}_{1})^{2}\cdot \varphi(P_{t_1}^{[1]}P_{t_2}^{[2]}(g)( {\bf y}_{1},{\bf y}_{2}) )^{2}\right)t_{1}\bigg|_{t_1=0}^{t_1=\infty} d{\bf y}_{1}\right.\\
&\qquad\left.+\iint_{M_{1}\times\mathbb{R}_{+}} \partial_{t_1}\left(P^{[1]}_{t_{1}}v_{{\bf y}_{2},t_{2}}({\bf y}_{1})^{2}\cdot \varphi(P_{t_1}^{[1]}P_{t_2}^{[2]}(g)( {\bf y}_{1},{\bf y}_{2}) )^{2}\right)dt_{1}d{\bf y}_{1}\right|\\
&\leq {\rm I_{11}}+{\rm I_{12}}+{\rm I_{13}}.
\end{align*}
For $ {\rm I_{11}}$, note that 
\begin{align*}
&t_{1}\nabla_{ M_1}\left(P^{[1]}_{t_{1}}v_{{\bf y}_{2},t_{2}}({\bf y}_{1})^{2}\cdot \varphi(P_{t_1}^{[1]}P_{t_2}^{[2]}(g)( {\bf y}_{1},{\bf y}_{2}) )^{2}\right)\\
&= 2P^{[1]}_{t_{1}}v_{{\bf y}_{2},t_{2}}({\bf y}_{1}) t_{1}\nabla_{ M_1} P^{[1]}_{t_{1}}v_{{\bf y}_{2},t_{2}}({\bf y}_{1})\cdot \varphi(P_{t_1}^{[1]}P_{t_2}^{[2]}(g)( {\bf y}_{1},{\bf y}_{2}) )^{2} \\
&\qquad+ P^{[1]}_{t_{1}}v_{{\bf y}_{2},t_{2}}({\bf y}_{1})^{2}\cdot 2\varphi(P_{t_1}^{[1]}P_{t_2}^{[2]}(g)( {\bf y}_{1},{\bf y}_{2}) ) t_{1}\nabla_{ M_1} \varphi(P_{t_1}^{[1]}P_{t_2}^{[2]}(g)( {\bf y}_{1},{\bf y}_{2}) ). 
\end{align*}
The size condition of the Poisson kernel in \eqref{P size} yields that 
$$t_{1}\nabla_{ M_1}\left(P^{[1]}_{t_{1}}v_{{\bf y}_{2},t_{2}}({\bf y}_{1})^{2}\cdot \varphi(P_{t_1}^{[1]}P_{t_2}^{[2]}(g)( {\bf y}_{1},{\bf y}_{2}) )^{2}\right)\bigg|_{d_1({\bf 0},{\bf y}_1)=\infty}=0.$$
Thus, $ {\rm I_{11}}=0$.

For $ {\rm I_{12}}$, we have
\begin{align*}
&\partial_{t_1}\left(P^{[1]}_{t_{1}}v_{{\bf y}_{2},t_{2}}({\bf y}_{1})^{2}\cdot \varphi(P_{t_1}^{[1]}P_{t_2}^{[2]}(g)( {\bf y}_{1},{\bf y}_{2}) )^{2}\right)t_{1}\\
&=2P^{[1]}_{t_{1}}v_{{\bf y}_{2},t_{2}}({\bf y}_{1})t_1\partial_{t_1}P^{[1]}_{t_{1}}v_{{\bf y}_{2},t_{2}}({\bf y}_{1})\cdot \varphi(P_{t_1}^{[1]}P_{t_2}^{[2]}(g)( {\bf y}_{1},{\bf y}_{2}) )^{2}\\
&\quad+ P^{[1]}_{t_{1}}v_{{\bf y}_{2},t_{2}}({\bf y}_{1})^{2} \cdot 2\varphi(P_{t_1}^{[1]}P_{t_2}^{[2]}(g)( {\bf y}_{1},{\bf y}_{2}) ) \varphi'(P_{t_1}^{[1]}P_{t_2}^{[2]}(g)( {\bf y}_{1},{\bf y}_{2}) ) \,t_1\partial_{t_1}P_{t_1}^{[1]}P_{t_2}^{[2]}(g)( {\bf y}_{1},{\bf y}_{2}).
\end{align*}
For $t_1\to\infty$, from the decay of Poisson kernel in \eqref{P size} and the functional calculus we see that the kernel of $t_1\partial_{t_1}P^{[1]}_{t_{1}}$
also satisfies a similar size condition as in \eqref{P size}. Thus,  
we have that 
\begin{align}\label{approx 0}
\partial_{t_1}\left(P^{[1]}_{t_{1}}v_{{\bf y}_{2},t_{2}}({\bf y}_{1})^{2}\cdot \varphi(P_{t_1}^{[1]}P_{t_2}^{[2]}(g)( {\bf y}_{1},{\bf y}_{2}) )^{2}\right)t_{1}\to 0.
\end{align}
On the other hand, when $t_1\to0^+$, we see that the term $t_1\partial_{t_1}P^{[1]}_{t_{1}}v_{{\bf y}_{2},t_{2}}({\bf y}_{1})\to 0$, since 
$t_1\partial_{t_1}P^{[1]}_{t_{1}}$ has integration zero. Thus, \eqref{approx 0} also holds when $t_1\to0^+$. This gives that $ {\rm I_{12}}=0$.

For $ {\rm I_{13}}$, we have
\begin{align*}
{\rm I_{13}}
&=\left|\int_{M_{1}}\left(|P^{[1]}_{t_{1}}v_{{\bf y}_{2},t_{2}}({\bf y}_{1})|^{2}\cdot \varphi(P_{t_1}^{[1]}P_{t_2}^{[2]}(g)( {\bf y}_{1},{\bf y}_{2}) )^{2}\right)\bigg|_{t_1=0}^{t_1=\infty} d{\bf y}_{1}\right|\\
&\leq \int_{M_{1}}|v_{{\bf y}_{2},t_{2}}({\bf y}_{1})|^{2}|\varphi(P_{t_{2}}^{[2]}(g)( {\bf y}_{1},{\bf y}_{2}))|^{2}d{\bf y}_{1},
\end{align*}
where the last inequality follows from the fact that $|P^{[1]}_{t_{1}}v_{{\bf y}_{2},t_{2}}({\bf y}_{1})|^{2}\cdot \varphi(P_{t_1}^{[1]}P_{t_2}^{[2]}(g)( {\bf y}_{1},{\bf y}_{2}) )^{2}\to0$ as $t_1\to\infty$ and that $P^{[1]}_{t_{1}} \to $ identity as $t_1\to 0^+$.

Therefore,
\begin{align}\label{I1 further}
{\rm I_{1}}&\leq {5\over9}\iiint_{\widetilde M\times \mathbb{R}_{+}}|v_{{\bf y}_{2},t_{2}}({\bf y}_{1})|^{2}|\varphi(P_{t_{2}}^{[2]}(g)( {\bf y}_{1},{\bf y}_{2}))|^{2}t_{2}dt_{2}d{\bf y}_{2}d{\bf y}_{1}\\
&= {5\over9}\int_{M_1}\iint_{M_2\times \mathbb{R}_{+}}|\nabla_{t_{2},M_{2}}P_{t_2}^{[2]}(f)( {\bf y}_{1},{\bf y}_{2})|^{2}|\varphi(P_{t_{2}}^{[2]}(g)( {\bf y}_{1},{\bf y}_{2}))|^{2}t_{2}dt_{2}d{\bf y}_{2}\ \ d{\bf y}_{1}.\nonumber
\end{align}
Using the same argument as in \eqref{domination}, we have
\begin{align}\label{domination for M2}
&\left|\nabla_{t_{2},M_{2}}P_{t_2}^{[2]}(f)( {\bf y}_{1},{\bf y}_{2})\right|^{2} \left| \varphi\Big(P_{t_2}^{[2]}(g)( {\bf y}_{1},{\bf y}_{2}) \Big)\right|^{2}\\[5pt]
&=- \frac{1}{2}\Delta_{t_{2},M_{2}}\left(P^{[2]}_{t_{2}} (f)( {\bf y}_{1},{\bf y}_{2})^{2}\cdot \varphi\Big(P_{t_2}^{[2]}(g)( {\bf y}_{1},{\bf y}_{2}) \Big)^{2}\right)\nonumber\\[5pt]
&\qquad-4 P^{[2]}_{t_{2}} (f)( {\bf y}_{1},{\bf y}_{2}) \nabla_{t_{2},M_{2}} P^{[2]}_{t_{2}} (f)( {\bf y}_{1},{\bf y}_{2})\nonumber \\[5pt]
&\hspace{1.5cm} \times \varphi\Big(P_{t_2}^{[2]}(g)( {\bf y}_{1},{\bf y}_{2}) \Big)\varphi^{\prime}\Big(P_{t_2}^{[2]}(g)( {\bf y}_{1},{\bf y}_{2}) \Big) \nabla_{t_{2},M_{2}} P_{t_2}^{[2]}(g)( {\bf y}_{1},{\bf y}_{2})\nonumber\\[5pt]
&\qquad-|P^{[2]}_{t_{2}} (f)( {\bf y}_{1},{\bf y}_{2})|^{2}\varphi^{\prime}\Big( P_{t_2}^{[2]}(g)( {\bf y}_{1},{\bf y}_{2}) \Big)^{2}\left|\nabla_{t_{2},M_{2}} P_{t_2}^{[2]}(g)( {\bf y}_{1},{\bf y}_{2})\right|^{2}\nonumber\\[5pt]
&\qquad-|P^{[2]}_{t_{2}} (f)( {\bf y}_{1},{\bf y}_{2})|^{2}\varphi\Big( P_{t_2}^{[2]}(g)( {\bf y}_{1},{\bf y}_{2})\Big)\varphi^{\prime\prime}\Big(P_{t_2}^{[2]}(g)( {\bf y}_{1},{\bf y}_{2})\Big)\left|\nabla_{t_{2},M_{2}} P_{t_2}^{[2]}(g)( {\bf y}_{1},{\bf y}_{2})\right|^{2}\nonumber\\[5pt]
&=:h_{1}( {\bf y}_{1},{\bf y}_{2},t_{2})+h_{2}( {\bf y}_{1},{\bf y}_{2},t_{2})+h_{3}( {\bf y}_{1},{\bf y}_{2},t_{2})+h_{4}( {\bf y}_{1},{\bf y}_{2},t_{2}).\nonumber
\end{align}
Then, by handling $h_2$ using the same estimate as we did for $f_2$, that is, we dominate it by 
\begin{align*}
&{1\over 10} \left|\nabla_{t_{2},M_{2}}P_{t_2}^{[2]}(f)( {\bf y}_{1},{\bf y}_{2})\right|^{2} \left| \varphi\Big(P_{t_2}^{[2]}(g)( {\bf y}_{1},{\bf y}_{2}) \Big)\right|^{2}\\
&\qquad+40\left|P^{[2]}_{t_{2}} (f)( {\bf y}_{1},{\bf y}_{2})\right|^2\left|\varphi^{\prime}\Big(P_{t_2}^{[2]}(g)( {\bf y}_{1},{\bf y}_{2}) \Big) \right|^2\left|\nabla_{t_{2},M_{2}} P_{t_2}^{[2]}(g)( {\bf y}_{1},{\bf y}_{2})\right|^2\\
&=:h_{21}( {\bf y}_{1},{\bf y}_{2},t_{2}) +h_{22}( {\bf y}_{1},{\bf y}_{2},t_{2}).
\end{align*}
Again, we see that 
$${5\over9}\int_{M_1}\iint_{M_2\times \mathbb{R}_{+}} h_{21}( {\bf y}_{1},{\bf y}_{2},t_{2}) t_{2}dt_{2}d{\bf y}_{2}\ d{\bf y}_{1}$$
can be absorbed by the right-hand side of \eqref{I1 further},
and we further have
\begin{align}\label{domination for M2 part2}
&
h_{22}( {\bf y}_{1},{\bf y}_{2},t_{2})+h_{3}( {\bf y}_{1},{\bf y}_{2},t_{2})+h_{4}( {\bf y}_{1},{\bf y}_{2},t_{2})\\
&\leq 40 |P_{t_2}^{[2]}(f)( {\bf y}_{1},{\bf y}_{2})|^2 \left|\Psi\Big(P_{t_2}^{[2]}(g)( {\bf y}_{1},{\bf y}_{2}) \Big)\right|^2 \left|\nabla_{t_{2},M_{2}} P_{t_2}^{[2]}(g)( {\bf y}_{1},{\bf y}_{2})\right|^2 \nonumber\\
&=: \mathfrak h_{2}( {\bf y}_{1},{\bf y}_{2},t_{2}),\nonumber
\end{align}
where we choose 
\begin{align*}
\Psi(t)=((\varphi^{\prime}(t))^{4}+(\varphi(t)\varphi^{\prime\prime}(t))^{2})^{1/4}.
\end{align*}

Then we further have
\begin{align*}
{\rm I_{1}}&\leq C\int_{M_1}\iint_{M_2\times \mathbb{R}_{+}}  h_{1}( {\bf y}_{1},{\bf y}_{2},t_{2}) t_{2}dt_{2}d{\bf y}_{2}\ \ d{\bf y}_{1}\\
&\qquad +C\int_{M_1}\iint_{M_2\times \mathbb{R}_{+}} \mathfrak h_{2}( {\bf y}_{1},{\bf y}_{2},t_{2}) t_{2}dt_{2}d{\bf y}_{2}\ \ d{\bf y}_{1}\\
&=:\widetilde {\rm I_{11}}+\widetilde {\rm I_{12}}
\end{align*}
with an  positive absolute constant $C$.
By repeating a similar integration by parts as in the estimates for ${\rm I}_1$, we have
\begin{align*}
\widetilde {\rm I_{11}}
&\leq C\iint_{\widetilde M}f( {\bf y}_{1},{\bf y}_{2})^{2}\varphi(g( {\bf y}_{1},{\bf y}_{2}))^{2}d{\bf y}_{1}d{\bf y}_{2}.
\end{align*}

It follows from the definition of the non-tangential maximal function $\mathcal{N}_{P}^{\beta}(f)$ that $f( {\bf y}_{1},{\bf y}_{2})\leq \mathcal{N}_{P}^{\beta}(f)( {\bf y}_{1},{\bf y}_{2})$. Besides, from the definition of the functions $g$ and $\varphi$, we see that for $( {\bf y}_{1},{\bf y}_{2})$ with $\varphi(g( {\bf y}_{1},{\bf y}_{2}))\neq 0$, we have that $g( {\bf y}_{1},{\bf y}_{2})>C_1$, which shows that $( {\bf y}_{1},{\bf y}_{2})\in E_\beta(\alpha)$. Hence, $\mathcal{N}_{P}^{\beta}(f)( {\bf y}_{1},{\bf y}_{2})\leq \alpha$. Thus,
\begin{align*}
\widetilde {\rm I_{11}}
&\leq C\iint_{E_\beta(\alpha)}\mathcal{N}_{P}^{\beta}(f)( {\bf y}_{1},{\bf y}_{2})^{2}d{\bf y}_{1}d{\bf y}_{2}.
\end{align*}

For $\widetilde {\rm I_{12}}$, again, 
from the definition of the functions $g$ and $\Psi$ (which inherits the support condition from $\varphi$), we see that for $( {\bf y}_{1},{\bf y}_{2},t_2)$ with $\Psi(P_{t_{2}}^{[2]}g( {\bf y}_{1},{\bf y}_{2}))\neq 0$, we have that $P_{t_{2}}^{[2]}g( {\bf y}_{1},{\bf y}_{2})>C_1$.
From
\eqref{P g small} we see that $( {\bf y}_{1},{\bf y}_{2},0,t_2)\in  \widetilde{W}_{\beta}$.
Hence, there exists ${\bf z}_2$ such that $({\bf y}_{1},{\bf z}_{2})\in E_\beta(\alpha)$ and that 
$|P_{t_2}^{[2]}(f)( {\bf y}_{1},{\bf y}_{2})|\leq \mathcal{N}_{P}^{\beta}(f)( {\bf y}_{1},{\bf z}_{2})\leq \alpha$.
Thus,
\begin{align*}
\widetilde {\rm I_{12}}&\leq C\iiint_{\widetilde M\times \mathbb{R}_{+}} |P_{t_2}^{[2]}(f)( {\bf y}_{1},{\bf y}_{2})|^2 \left|\Psi\Big(P_{t_2}^{[2]}(g)( {\bf y}_{1},{\bf y}_{2}) \Big)\right|^2 \left|\nabla_{t_{2},M_{2}} P_{t_2}^{[2]}(g)( {\bf y}_{1},{\bf y}_{2})\right|^2 t_{2}dt_{2}d{\bf y}_{2}\ \ d{\bf y}_{1}\\
&\leq C\alpha^2\iiint_{\widetilde M\times \mathbb{R}_{+}}   \left|\nabla_{t_{2},M_{2}} P_{t_2}^{[2]}(g)( {\bf y}_{1},{\bf y}_{2})\right|^2 t_{2}dt_{2}d{\bf y}_{2}\ \ d{\bf y}_{1}\\
&= C\alpha^2\iiint_{\widetilde M\times \mathbb{R}_{+}}   \left|t_{2}\nabla_{t_{2},M_{2}} P_{t_2}^{[2]}(g)( {\bf y}_{1},{\bf y}_{2})\right|^2 {dt_{2}\over t_{2}}d{\bf y}_{2}\ \ d{\bf y}_{1}\\
&= C\alpha^2\iiint_{\widetilde M\times \mathbb{R}_{+}}   \left|t_{2}\nabla_{t_{2},M_{2}} P_{t_2}^{[2]}(1-g)( {\bf y}_{1},{\bf y}_{2})\right|^2 {dt_{2}\over t_{2}}d{\bf y}_{2}\ \ d{\bf y}_{1},
\end{align*}
where the last equality follows from the fact that the kernel of $t_{2}\nabla_{t_{2},M_{2}} P_{t_2}^{[2]}$ has integration zero.
Therefore, by using the Littlewood--Paley estimate, we obtain that
\begin{align*}
\widetilde {\rm I_{12}}&\leq C\alpha^2\|1-g\|_{L^2(\widetilde M)}^2 = C\alpha^2\left|E_\beta(\alpha)^c\right|.
\end{align*}
This finishes the estimate of the term ${\rm I_{1}}$.  We now term to ${\rm I_{2}}$.
By noting that 
\begin{align*}
\nabla_{t_{1},M_{1}}\Phi\Big(P_{t_1}^{[1]}P_{t_2}^{[2]}(g)( {\bf y}_{1},{\bf y}_{2}) \Big)= \Phi'\Big(P_{t_1}^{[1]}P_{t_2}^{[2]}(g)( {\bf y}_{1},{\bf y}_{2}) \Big) \nabla_{t_{1},M_{1}} P_{t_1}^{[1]}P_{t_2}^{[2]}(g)( {\bf y}_{1},{\bf y}_{2}),
\end{align*}
we have
\begin{align}\label{a2}
{\rm I_{2}}&= {400\over 9}\iiiint_{\widetilde M\times \mathbb{R}_{+}\times\mathbb{R}_{+}} \left| \nabla_{t_{2},M_{2}}P_{t_2}^{[2]}P^{[1]}_{t_{1}}(f)( {\bf y}_{1},{\bf y}_{2})\right|^2\\ 
&\qquad\qquad\times \left|\nabla_{t_{1},M_{1}} \Phi\Big(P_{t_1}^{[1]}P_{t_2}^{[2]}(g)( {\bf y}_{1},{\bf y}_{2})\Big)\right|^2t_{1}t_{2}dt_{1}dt_{2}d{\bf y}_{1}d{\bf y}_{2}.\nonumber
\end{align}
Observe that
\begin{align*}
&|\nabla_{t_{2},M_{2}}P_{t_{1}}^{[1]}P_{t_{2}}^{[2]}(f)( {\bf y}_1,{\bf y}_2)|^{2}|\nabla_{t_{1},M_{1}}\Phi(P_{t_{1}}^{[1]}P_{t_{2}}^{[2]}(g)( {\bf y}_1,{\bf y}_2))|^{2}\\[5pt]
&=-\frac{1}{2}\Delta_{t_{2},M_{2}}\left(P_{t_{1}}^{[1]} P_{t_{2}}^{[2]} (f)( {\bf y}_1,{\bf y}_2)^{2}|\nabla_{t_{1},M_{1}}\Phi(P_{t_{1}}^{[1]}P_{t_{2}}^{[2]}(g)( {\bf y}_1,{\bf y}_2))|^{2}\right)\\[5pt]
&\qquad-4P_{t_{1}}^{[1]} P_{t_{2}}^{[2]}(f)( {\bf y}_1,{\bf y}_2) \nabla_{t_{2},M_{2}}P_{t_{1}}^{[1]} P_{t_{2}}^{[2]}(f)( {\bf y}_1,{\bf y}_2)\\[5pt]
&\hskip2cm\times\nabla_{t_{1},M_{1}}\Phi( P_{t_{1}}^{[1]} P_{t_{2}}^{[2]}(g)( {\bf y}_1,{\bf y}_2))\nabla_{t_{2},M_{2}}\nabla_{t_{1},M_{1}}\Phi(P_{t_{1}}^{[1]} P_{t_{2}}^{[2]}(g)( {\bf y}_1,{\bf y}_2))\\[5pt]
&\qquad-|P_{t_{1}}^{[1]}P_{t_{2}}^{[2]}f( {\bf y}_1,{\bf y}_2)|^{2}|\nabla_{t_{2},M_{2}}\nabla_{t_{1},M_{1}}\Phi(P_{t_{1}}^{[1]} P_{t_{2}}^{[2]}(g)( {\bf y}_1,{\bf y}_2))|^{2}\\[5pt]
&\qquad-|P_{t_{1}}^{[1]} P_{t_{2}}^{[2]}(f)( {\bf y}_1,{\bf y}_2)|^{2}\\[5pt]
&\hskip2cm\times\nabla_{t_{1},M_{1}}\Phi(P_{t_{1}}^{[1]} P_{t_{2}}^{[2]}(g)( {\bf y}_1,{\bf y}_2))\nabla_{t_{2},M_{2}}\nabla_{t_{2},M_{2}}\nabla_{t_{1},M_{1}}\Phi(P_{t_{1}}^{[1]} P_{t_{2}}^{[2]}(g)( {\bf y}_1,{\bf y}_2))\\[5pt]
&=:\mathfrak F_{1}( {\bf y}_1,{\bf_y}_2,t_{1},t_{2})+\mathfrak F_{2}( {\bf y}_1,{\bf y}_2,t_{1},t_{2})+\mathfrak F_{3}( {\bf y}_1,{\bf y}_2,t_{1},t_{2})+\mathfrak F_{4}( {\bf y}_1,{\bf y}_2,t_{1},t_{2}).
\end{align*}
Thus, the right-hand side of \eqref{a2} is bounded by ${\rm II_{21}+II_{22}+II_{23}+II_{24}}$, where
\begin{align*}
{\rm II}_{2j}:=C\left|\iiiint_{\widetilde M\times \mathbb{R}_{+}\times\mathbb{R}_{+}}\mathfrak F_{j}( {\bf y}_1,{\bf y}_2,t_{1},t_{2})t_{1}t_{2}dt_{1}dt_{2}d{\bf y}_{1}d{\bf y}_{2}\right|,\quad j=1,2,3,4.
\end{align*}

To estimate the term ${\rm II_{21}}$, we first let $\Phi_{1}(t)$ be a smooth function on $\mathbb{R}$ such that $$\Phi_{1}^{\prime}(t)=(\Phi'(t)^{4}+(\Phi(t)\Phi^{\prime\prime}(t))^{2})^{\frac{1}{4}}.$$
Before we move on, note that 
again, 
from the definition of the functions $g$ and $\Phi_{1}^{\prime}$  (which inherits the support condition from $\Phi$), we see that for $( {\bf y}_{1},{\bf y}_{2},t_1)$ with $\Phi'_1(P_{t_{1}}^{[1]}g( {\bf y}_{1},{\bf y}_{2}))\neq 0$, we have that $P_{t_{1}}^{[1]}g( {\bf y}_{1},{\bf y}_{2})>C_1$.
From
\eqref{P g small} we see that $( {\bf y}_{1},{\bf y}_{2},t_1,0)\in  \widetilde{W}_{\beta}$.
Hence, there exists ${\bf z}_1$ such that $({\bf z}_{1},{\bf y}_{2})\in E_\beta(\alpha)$ and that 
$|P_{t_1}^{[1]}(f)( {\bf y}_{1},{\bf y}_{2})|\leq \mathcal{N}_{P}^{\beta}(f)( {\bf z}_{1},{\bf y}_{2})\leq \alpha$.

Next, by repeating a similar integration by parts, we have
\begin{align*}
{\rm II_{21}}
&\leq C\iint_{\widetilde M}f( {\bf y}_{1},{\bf y}_{2})^{2}\Phi(g( {\bf y}_{1},{\bf y}_{2}))^{2}d{\bf y}_{1}d{\bf y}_{2}\\
&\quad+ C\iiint_{\widetilde M\times \mathbb{R}_{+}}P_{t_{1}}^{[1]}(f)( {\bf y}_1,{\bf y}_2)^{2}|\nabla_{t_{1},M_{1}}\Phi_{1}(P_{t_{1}}^{[1]}(g)( {\bf y}_1,{\bf y}_2))|^{2}t_{1}dt_{1}d{\bf y}_{1}d{\bf y}_{2}\\
&\leq C\iint_{E_\beta(\alpha)}\mathcal{N}_{P}^{\beta}(f)( {\bf y}_{1},{\bf y}_{2})^{2}d{\bf y}_{1}d{\bf y}_{2}\\
&\quad+ C\alpha^{2}\iiint_{\widetilde M\times \mathbb{R}_{+}}|\nabla_{t_{1},M_{1}}(P_{t_{1}}^{[1]}(g)( {\bf y}_1,{\bf y}_2))|^{2}t_{1}dt_{1}d{\bf y}_{1}d{\bf y}_{2}\\
&= C\iint_{E_\beta(\alpha)}\mathcal{N}_{P}^{\beta}(f)( {\bf y}_{1},{\bf y}_{2})^{2}d{\bf y}_{1}d{\bf y}_{2}\\
&\quad+ C\alpha^{2}\iiint_{\widetilde M\times \mathbb{R}_{+}}|(t_{1}\nabla_{t_{1},M_{1}}P_{t_{1}}^{[1]}(1-g)( {\bf y}_1,{\bf y}_2))|^{2}\frac{dt_{1}d{\bf y}_{1}d{\bf y}_{2}}{t_{1}}\\
&\leq C\iint_{E_\beta(\alpha)}\mathcal{N}_{P}^{\beta}(f)( {\bf y}_{1},{\bf y}_{2})^{2}d{\bf y}_{1}d{\bf y}_{2}+
C\alpha^{2}\|1-g\|_{L^{2}(\widetilde M)}^{2}\\
&=C\iint_{E_\beta(\alpha)}\mathcal{N}_{P}^{\beta}(f)( {\bf y}_{1},{\bf y}_{2})^{2}d{\bf y}_{1}d{\bf y}_{2}+C\alpha^{2}|E_\beta(\alpha)^c|,
\end{align*}
where in the second inequality we used the chain rule
\begin{align*}
\nabla_{t_{1},M_{1}}\Phi_{1}(P_{t_{1}}^{[1]}(g)( {\bf y}_1,{\bf y}_2))=\Phi_{1}^{\prime}(P_{t_{1}}^{[1]}(g)( {\bf y}_1,{\bf y}_2))\nabla_{t_{1},M_{1}}(P_{t_{1}}^{[1]}(g)( {\bf y}_1,{\bf y}_2))
\end{align*}
and the inequality $|\Phi_{1}^{\prime}(P_{t_{1}}^{[1]}(g)( {\bf y}_1,{\bf y}_2))|\leq C$ as well as the support condition on $\Phi_{1}^{\prime}$.

For the term ${\rm II_{22}}$, we apply Young's inequality to see that
\begin{align*}
{\rm II_{22}}&\leq \frac{1}{10}\iiiint_{\widetilde M\times \mathbb{R}_{+}\times\mathbb{R}_{+}}|\nabla_{t_{2},M_{2}}P_{t_{1}}^{[1]} P_{t_{2}}^{[2]}(f)( {\bf y}_1,{\bf y}_2)|^{2}\\
&\hspace{3.9cm} \times |\nabla_{t_{1},M_{1}}\Phi(P_{t_{1}}^{[1]} P_{t_{2}}^{[2]}(g)( {\bf y}_1,{\bf y}_2))|^{2}t_{1}t_{2}dt_{1}dt_{2}d{\bf y}_{1}d{\bf y}_{2}\\
&\qquad+C\iiiint_{\widetilde M\times \mathbb{R}_{+}\times\mathbb{R}_{+}}|P_{t_{1}}^{[1]} P_{t_{2}}^{[2]}(f)( {\bf y}_1,{\bf y}_2)|^{2}\\
&\hspace{3.9cm} \times |\nabla_{t_{2},M_{2}}\nabla_{t_{1},M_{1}}\Phi( P_{t_{1}}^{[1]} P_{t_{2}}^{[2]}(g)( {\bf y}_1,{\bf y}_2))|^{2}t_{1}t_{2}dt_{1}dt_{2}d{\bf y}_{1}d{\bf y}_{2}\\
&=:{\rm II_{221}}+{\rm II_{222}}.
\end{align*}
Since ${\rm II_{221}}$ can be absorded by ${\rm I}_{2}$, it suffices to estimate the term ${\rm II_{222}}$. By the chain rule,
\begin{align*}
&\nabla_{t_{2},M_{2}}\nabla_{t_{1},M_{1}}\Phi(P_{t_{1}}^{[1]} P_{t_{2}}^{[2]}(g)( {\bf y}_1,{\bf y}_2))\\
&=\Phi^{\prime\prime}(P_{t_{1}}^{[1]} P_{t_{2}}^{[2]}(g)( {\bf y}_1,{\bf y}_2))\nabla_{t_{2},M_{2}}P_{t_{1}}^{[1]} P_{t_{2}}^{[2]}(g)( {\bf y}_1,{\bf y}_2)\nabla_{t_{1},M_{1}} P_{t_{1}}^{[1]} P_{t_{2}}^{[2]}(g)( {\bf y}_1,{\bf y}_2)\\
&\qquad +\Phi^{\prime}(P_{t_{1}}^{[1]} P_{t_{2}}^{[2]}(g)( {\bf y}_1,{\bf y}_2))\nabla_{t_{2},M_{2}}\nabla_{t_{1},M_{1}}P_{t_{1}}^{[1]} P_{t_{2}}^{[2]}(g)( {\bf y}_1,{\bf y}_2).
\end{align*}
Then ${\rm II_{222}}$ can be further bounded by ${\rm II_{2221}}$ and ${\rm II_{2222}}$ with the above two integrands respectively.

Denote $\mathcal M_{1}$ and $\mathcal M_{2}$ be the Hardy--Littlewood maximal functions on $M_{1}$ and $M_{2}$, respectively. Then by the support property of $\Phi^{\prime\prime}$,
\begin{align*}
{\rm II_{2221}}&= \iiiint_{\widetilde M\times \mathbb{R}_{+}\times\mathbb{R}_{+}}|P_{t_{1}}^{[1]}P_{t_{2}}^{[2]}(f)( {\bf y}_1,{\bf y}_2)|^{2}|\Phi^{\prime\prime}(P_{t_{1}}^{[1]}P_{t_{2}}^{[2]}(g)( {\bf y}_1,{\bf y}_2))|^{2}\\[5pt]
& \hspace{0.8cm}\times |\nabla_{t_{2},M_{2}} P_{t_{1}}^{[1]}P_{t_{2}}^{[2]}(g)( {\bf y}_1,{\bf y}_2)|^{2}|\nabla_{t_{1},M_{1}} P_{t_{1}}^{[1]}P_{t_{2}}^{[2]}(g)( {\bf y}_1,{\bf y}_2)|^{2}t_{1}t_{2}dt_{1}dt_{2}d{\bf y}_{1}d{\bf y}_{2}\\[5pt]
&\leq C\alpha^{2} \iiiint_{\widetilde M\times \mathbb{R}_{+}\times\mathbb{R}_{+}}|\nabla_{t_{2},M_{2}} P_{t_{1}}^{[1]}P_{t_{2}}^{[2]}(g)( {\bf y}_1,{\bf y}_2)|^{2}\\[5pt]
& \hspace{5.2cm}\times |\nabla_{t_{1},M_{1}} P_{t_{1}}^{[1]}P_{t_{2}}^{[2]}(g)( {\bf y}_1,{\bf y}_2)|^{2}t_{1}t_{2}dt_{1}dt_{2}d{\bf y}_{1}d{\bf y}_{2}\\[5pt]
&\leq C\alpha^{2} \iiiint_{\widetilde M\times \mathbb{R}_{+}\times\mathbb{R}_{+}}\left|\mathcal M_{1}\left(\big|t_{2}\nabla_{t_{2},M_{2}}P^{[2]}_{t_{2}}(g)\big|\right)( {\bf y}_1,{\bf y}_2)\right|^{2}\\[5pt]
& \hspace{5.2cm}\times \left|\mathcal M_{2}\left(\big|t_{1}\nabla_{t_{1},M_{1}}P^{[1]}_{t_{1}}(g)\big|\right)( {\bf y}_1,{\bf y}_2)\right|^{2}\frac{dt_{1}dt_{2}d{\bf y}_{1}d{\bf y}_{2}}{t_{1}t_{2}}.
\end{align*}
Applying H\"{o}lder's inequality, we further have
\begin{align*}
{\rm II_{2221}}
&\leq C\alpha^{2} \left(\iint_{\widetilde M}\left(\int_{\mathbb{R}_{+}}\left|\mathcal M_{1}\left(\big|t_{2}\nabla_{t_{2},M_{2}}P^{[2]}_{t_{2}}(g)\big|\right)( {\bf y}_1,{\bf y}_2)\right|^{2}\frac{dt_{2}}{t_{2}}\right)^{2}d{\bf y}_{1}d{\bf y}_{2}\right)^{1/2}\\
& \hspace{2.5cm}\times \left(\iint_{\widetilde M}\left(\int_{\mathbb{R}_{+}}\left|\mathcal M_{2}\left(\big|t_{1}\nabla_{t_{1},M_{1}}P^{[1]}_{t_{1}}(g)\big|\right)( {\bf y}_1,{\bf y}_2)\right|^{2}\frac{dt_{1}}{t_{1}}\right)^{2}d{\bf y}_{1}d{\bf y}_{2}\right)^{1/2}\\
&\leq C\alpha^{2} \left(\iint_{\widetilde M}\left(\int_{\mathbb{R}_{+}}\left|t_{2}\nabla_{t_{2},M_{2}}P^{[2]}_{t_{2}}(g)( {\bf y}_1,{\bf y}_2)\right|^{2}\frac{dt_{2}}{t_{2}}\right)^{2}d{\bf y}_{1}d{\bf y}_{2}\right)^{1/2}\\
& \hspace{2.5cm}\times \left(\iint_{\widetilde M}\left(\int_{\mathbb{R}_{+}}\left|t_{1}\nabla_{t_{1},M_{1}}P^{[1]}_{t_{1}}(g)( {\bf y}_1,{\bf y}_2)\right|^{2}\frac{dt_{1}}{t_{1}}\right)^{2}d{\bf y}_{1}d{\bf y}_{2}\right)^{1/2}\\
&= C\alpha^{2} \left(\iint_{\widetilde M}\left(\int_{\mathbb{R}_{+}}\left|t_{2}\nabla_{t_{2},M_{2}}P^{[2]}_{t_{2}}(1-g)( {\bf y}_1,{\bf y}_2)\right|^{2}\frac{dt_{2}}{t_{2}}\right)^{2}d{\bf y}_{1}d{\bf y}_{2}\right)^{1/2}\\
& \hspace{2.5cm}\times \left(\iint_{\widetilde M}\left(\int_{\mathbb{R}_{+}}\left|t_{1}\nabla_{t_{1},M_{1}}P^{[1]}_{t_{1}}(1-g)( {\bf y}_1,{\bf y}_2)\right|^{2}\frac{dt_{1}}{t_{1}}\right)^{2}d{\bf y}_{1}d{\bf y}_{2}\right)^{1/2}\\
&\leq C\alpha^{2} \|1-g\|_{L^{2}(\widetilde M)}^{2}\\
&=C\alpha^{2}\left|E_\beta(\alpha)^c\right|,
\end{align*}
where the second inequality we used the vector-valued inequality for the Hardy--Littlewood maximal functions and the last inequality we applied the Littlewood--Paley theory. Next it follows from the support condition on $\Phi^{\prime}$ that for $( {\bf y}_1,{\bf y}_2,t_1,t_2)$ satisfying $\Phi^{\prime}(P_{t_{1}}^{[1]}P_{t_{2}}^{[2]}(g)( {\bf y}_1,{\bf y}_2))\not=0$, we have $|P_{t_{1}}^{[1]}P_{t_{2}}^{[2]}(f)( {\bf y}_1,{\bf y}_2)|\leq\alpha$.

As a consequence, we have
\begin{align*}
{\rm II_{2222}}&=\iiiint_{\widetilde M\times \mathbb{R}_{+}\times\mathbb{R}_{+}}|P_{t_{1}}^{[1]}P_{t_{2}}^{[2]}(f)( {\bf y}_1,{\bf y}_2)|^{2}
|\Phi^{\prime}(P_{t_{1}}^{[1]}P_{t_{2}}^{[2]}(g)( {\bf y}_1,{\bf y}_2))|^{2}\\
&\hspace{1.0cm} \times |\nabla_{t_{2},M_{2}}\nabla_{t_{1},M_{1}}P_{t_{1}}^{[1]}P_{t_{2}}^{[2]}(g)( {\bf y}_1,{\bf y}_2)|^{2}t_{1}t_{2}dt_{1}dt_{2}d{\bf y}_{1}d{\bf y}_{2}\\
&\leq C\alpha^{2}\iiiint_{\widetilde M\times \mathbb{R}_{+}\times\mathbb{R}_{+}}|t_{1}\nabla_{t_{1},M_{1}}P^{[1]}_{t_{1}}t_{2}\nabla_{t_{2},M_{2}}P^{[2]}_{t_{2}}(g)( {\bf y}_1,{\bf y}_2)|^{2}\frac{dt_{1}dt_{2}d{\bf y}_{1}d{\bf y}_{2}}{t_{1}t_{2}}\\
&= C\alpha^{2}\iiiint_{\widetilde M\times \mathbb{R}_{+}\times\mathbb{R}_{+}}|t_{1}\nabla_{t_{1},M_{1}}P^{[1]}_{t_{1}}t_{2}\nabla_{t_{2},M_{2}}P^{[2]}_{t_{2}}(1-g)( {\bf y}_1,{\bf y}_2)|^{2}\frac{dt_{1}dt_{2}d{\bf y}_{1}d{\bf y}_{2}}{t_{1}t_{2}}\\
&\leq C\alpha^{2}\|1-g\|_{L^{2}(\widetilde M)}^{2}\\
&=C\alpha^{2}\left| E_\beta(\alpha)^c\right |,
\end{align*}
where in the second inequality we applied the Littlewood--Paley theory again.

Similar to the  estimate of ${\rm II_{222}}$, we obtain that
\begin{align*}
{\rm II_{23}}\leq C\alpha^{2}\left|\{( {\bf y}_1,{\bf y}_2)\in\widetilde M:\ \mathcal{N}_{P}^{\beta}(f)( {\bf y}_1,{\bf y}_2)>\alpha\}\right |.
\end{align*}

Finally, we turn to estimate ${\rm II_{24}}$. By the chain rule,
\begin{align*}
&\nabla_{t_{1},M_{1}}\Phi(P_{t_{1}}^{[1]}P_{t_{2}}^{[2]}(g)( {\bf y}_1,{\bf y}_2))\nabla_{t_{2},M_{2}}\nabla_{t_{2},M_{2}}\nabla_{t_{1},M_{1}}\Phi(P_{t_{1}}^{[1]}P_{t_{2}}^{[2]}(g)( {\bf y}_1,{\bf y}_2))\\
&=(\Phi^{\prime}\Phi^{\prime\prime\prime})(P_{t_{1}}^{[1]}P_{t_{2}}^{[2]}(g)( {\bf y}_1,{\bf y}_2))|\nabla_{t_{1},M_{1}}P_{t_{1}}^{[1]}P_{t_{2}}^{[2]}(g)( {\bf y}_1,{\bf y}_2)|^{2}|\nabla_{t_{2},M_{2}}P_{t_{1}}^{[1]}P_{t_{2}}^{[2]}(g)( {\bf y}_1,{\bf y}_2)|^{2}\\
&\hspace{0.5cm} +2(\Phi^{\prime}\Phi^{\prime\prime})(P_{t_{1}}^{[1]}P_{t_{2}}^{[2]}(g)( {\bf y}_1,{\bf y}_2))\nabla_{t_{1},M_{1}}P_{t_{1}}^{[1]}P_{t_{2}}^{[2]}(g)( {\bf y}_1,{\bf y}_2)\nabla_{t_{2},M_{2}}P_{t_{1}}^{[1]}P_{t_{2}}^{[2]}(g)( {\bf y}_1,{\bf y}_2)\\
&\hspace{1cm} \times \nabla_{t_{1},M_{1}}\nabla_{t_{2},M_{2}}P_{t_{1}}^{[1]}P_{t_{2}}^{[2]}(g)( {\bf y}_1,{\bf y}_2).
\end{align*}
Thus, ${\rm II_{24}}$ can be dominated by ${\rm II_{241}}+{\rm II_{242}}$ with respect to the above two terms in the integrand respectively. Using the same argument, we get that
\begin{align*}
{\rm II_{241}}\leq C\alpha^{2}\left| E_\beta(\alpha)^c\right |.
\end{align*}
Next, it follows from the support property of $\Phi^{\prime}\Phi^{\prime\prime}$ and H\"{o}lder's inequality that
\begin{align*}
{\rm I_{242}}&\leq C\alpha^{2}\iiiint_{\widetilde M\times \mathbb{R}_{+}\times\mathbb{R}_{+}} |\nabla_{t_{1},M_{1}}P_{t_{1}}^{[1]}P_{t_{2}}^{[2]}(g)( {\bf y}_1,{\bf y}_2)|\,|\nabla_{t_{2},M_{2}}P_{t_{1}}^{[1]}P_{t_{2}}^{[2]}(g)( {\bf y}_1,{\bf y}_2)|\\
&\hspace{1cm} \times |\nabla_{t_{1},M_{1}}\nabla_{t_{2},M_{2}}P_{t_{1}}^{[1]}P_{t_{2}}^{[2]}(g)( {\bf y}_1,{\bf y}_2)|t_{1}t_{2}dt_{1}dt_{2}d{\bf y}_{1}d{\bf y}_{2}\\
&\leq C\alpha^{2}\bigg(\iiiint_{\widetilde M\times \mathbb{R}_{+}\times\mathbb{R}_{+}} |\nabla_{t_{1},M_{1}}P_{t_{1}}^{[1]}P_{t_{2}}^{[2]}(g)( {\bf y}_1,{\bf y}_2)|^{2}\\
&\hspace{5.1cm} \times |\nabla_{t_{2},M_{2}}P_{t_{1}}^{[1]}P_{t_{2}}^{[2]}(g)( {\bf y}_1,{\bf y}_2)|^{2}t_{1}t_{2}dt_{1}dt_{2}d{\bf y}_{1}d{\bf y}_{2}\bigg)\\
&\qquad+ C\alpha^{2}\left(\iiiint_{\widetilde M\times \mathbb{R}_{+}\times\mathbb{R}_{+}} |\nabla_{t_{1},M_{1}}\nabla_{t_{2},M_{2}}P_{t_{1}}^{[1]}P_{t_{2}}^{[2]}(g)( {\bf y}_1,{\bf y}_2)|^{2}t_{1}t_{2}dt_{1}dt_{2}d{\bf y}_{1}d{\bf y}_{2}\right)\\
&\leq C\alpha^{2} \left| E_\beta(\alpha)^c\right |,
\end{align*}
where the last inequality follows from the estimates of the terms ${\rm II_{2221}}$ and ${\rm II_{2222}}$, that is, via taking the vector-valued Hardy--Littlewood maximal function estimate and then the Littlewood--Paley estimates.

Combining these estimates together, we conclude that
\begin{align*}
\iint_{A_{\beta}(\alpha)}S_{P}(f)( {\bf x}_{1},{\bf x}_{2})^2d{\bf x}_{1}d{\bf x}_{2}
\leq C\iint_{E_\beta(\alpha)}\mathcal{N}_{P}^{\beta}(f)( {\bf y}_{1},{\bf y}_{2})^{2}d{\bf y}_{1}d{\bf y}_{2}+C\alpha^{2} \left|E_\beta(\alpha)^c\right |.
\end{align*}
Hence, combining \eqref{good lambda} and the above estimate, we see that 
\begin{align*}
&\left|\{( {\bf y}_1,{\bf y}_2)\in \widetilde M:\ S_{P}(f)( {\bf y}_1,{\bf y}_2)>\alpha\}\right|\\
&\leq C\left|\{( {\bf y}_1,{\bf y}_2)\in\widetilde M:\ \mathcal{N}_{P}^{\beta}(f)( {\bf y}_1,{\bf y}_2)>\alpha\}\right |+\frac{C}{\alpha^{2}}\iint_{E_\beta(\alpha)}\mathcal{N}_{P}^{\beta}(f)( {\bf y}_1,{\bf y}_2)^{2}d{\bf y}_1d{\bf y}_{2}.
\end{align*}

The proof of Theorem \ref{main thm} is complete.

\section{Applications}\label{Sec4}
\setcounter{equation}{0}

We address two direct applications of Theorem \ref{main thm}.

\subsection{Weak endpoint estimate for Cauchy--Szeg\H{o} projection on $\widetilde M$}

We recall that in the well-known result of Diaz \cite{Diaz} (see also recent result \cite{CLTW}),  the explicit pointwise size estimate and regularity estimate of the Cauchy--Szeg\H{o} kernel on $M$ is given as follows.

\noindent {\bf Theorem A.}\ \ 
{\it For ${\bf x}$ and ${\bf y}$ in $M$ with ${\bf x}\not={\bf y}$,  the Cauchy--Szeg\H{o} projection ${\bf S}$ associated with the kernel  $S({\bf x},{\bf y})$ is a Calder\'on--Zygmund operator, i.e.,
\begin{align*}
|S({\bf x},{\bf y})|
\approx{1\over V({\bf x},{\bf y})};
\end{align*}
there is $\epsilon>0$ such that for ${\bf x}\not={\bf x}'$ and
for $ d({\bf x},{\bf x}')\leq  c d({\bf x},{\bf y})$
with some small positive constant  $c$,
\begin{align*}
|S({\bf x},{\bf y})-S({\bf x}',{\bf y})|
 \leq C_1  { 1\over V({\bf x},{\bf y}) } \ \bigg({d({\bf x},{\bf x}')\over d({\bf x},{\bf y})} \bigg) ^{\epsilon}; 
 \end{align*}
 for ${\bf y}\not={\bf y}'$ and
for $ d({\bf y},{\bf y}')\leq  c d({\bf x},{\bf y})$
with some small positive constant  $c$,
\begin{align*} 
|S({\bf x},{\bf y})-S({\bf x},{\bf y}')|
 \leq C_1  { 1\over V({\bf x},{\bf y}) } \ \bigg({d({\bf y},{\bf y}')\over d({\bf x},{\bf y})} \bigg) ^{\epsilon}  
 \end{align*}
with some constant $C_1>0$.
}

Consider $\widetilde M=M_1\times M_2$. The Cauchy--Szeg\H{o}
projection  ${\widetilde {\bf S}}={\bf S}_1\circ {\bf S}_2$ is a product Calder\'on--Zygmund operator on $L^2(\widetilde M)$, where ${\bf S}_j$ is the Cauchy--Szeg\H{o}
projection  on $M_j$ for $j=1,2$. The general framework of 
product Calder\'on--Zygmund operator on space of homogeneous type was studied in \cite{HLLin}.

Following the framework in \cite{CLLP}, we see that Theorem \ref{main thm} gives rise to the weak endpoint estimate of $S_P(f)({\bf x}_{1},{\bf x}_{2})$. To be more explicit, following \cite{CoFe} we first see that there is $C>0$ such that for all $\lambda>0$,
$$ |\{ ({\bf x}_{1},{\bf x}_{2})\in\widetilde M: \mathcal M(f)({\bf x}_{1},{\bf x}_{2})>\lambda\}| \leq C\|\lambda^{-1}f\|_{L \log L(\widetilde M )},$$
where $$\|f\|_{L \log L(\widetilde M )}= \iint_{\widetilde M} |f({\bf x}_{1},{\bf x}_{2})|\log(e+|f({\bf x}_{1},{\bf x}_{2})|)d{\bf x}_{1}d{\bf x}_{2}.$$
Then Theorem \ref{main thm} yields that 
$ |\{ ({\bf x}_{1},{\bf x}_{2})\in\widetilde M: S_P(f)({\bf x}_{1},{\bf x}_{2})>\lambda\}| \leq C\|\lambda^{-1}f\|_{L \log L(\widetilde M )},$
and hence we have
\begin{align}
|\{ ({\bf x}_{1},{\bf x}_{2})\in\widetilde M: {\widetilde {\bf S}}(f)({\bf x}_{1},{\bf x}_{2})>\lambda\}| \leq C\|\lambda^{-1}f\|_{L \log L(\widetilde M )}.
\end{align}

\subsection{Maximal function characterisation for product Hardy space on $\widetilde M$}

For the sake of simplicity, we take Hardy space $H^1(\widetilde M)$ as an application. The same argument holds for  $H^p(\widetilde M)$ for $p<1$ but with a more complicated but very standard definition via distributions. 

We now define 
\begin{align}
H_{S_P}^1(\widetilde M)= \{f\in L^1(\widetilde M): S_P(f)\in L^1(\widetilde M)\}.
\end{align}
Following the Plancherel--Polya inequality in \cite{HLL2}, we see that $H_{S_P}^1(\widetilde M)$ is equivalent to the Hardy space $H^1(\widetilde M)$ in \cite{HLL2} given via discrete square function (or discrete area functions) using the discrete reproducing formula via approximation to identity.  Thus, $H_{S_P}^1(\widetilde M)$ has atomic decomposition, which we refer to \cite{HLPW}.

Next, we define 
\begin{align}
H_{max}^1(\widetilde M)= \{f\in L^1(\widetilde M): \mathcal{N}_P^{\beta}(f)\in L^1(\widetilde M)\}.
\end{align}

Then we have the following equivalence characterisation.
\begin{prop}\label{prop max}
$H_{S_P}^1(\widetilde M)$ coincides with $H_{max}^1(\widetilde M)$ and they have equivalent norms.
\end{prop}
\begin{proof}
Suppose $f\in H_{S_P}^1(\widetilde M)\cap L^2(\widetilde M)$, then following \cite{HLPW} $f$ has an atomic decomposition $f=\sum_j\lambda_ja_j$ where each $a_j$ is a product atom and $\sum_j|\lambda_j|\approx \|f\|_{H_{S_P}^1(\widetilde M)}$.
Thus, it suffices to show that $\|\mathcal{N}_P^{\beta}(a_j)\|_{ L^1(\widetilde M)}\leq C$ for every atom $a_j$, where $C$ is an absolute constant. This is a standard argument, which follows from the size and cancellation of $a_j$,  the size and regularity estimates of the Poisson kernels $P_{t_1}^{[1]}$ and $P_{t_2} ^{[2]}$, and Journ\'e's covering lemma \cite{P}. Thus, we have 
$\|f\|_{H_{max}^1(\widetilde M)}\leq C \|f\|_{H_{S_P}^1(\widetilde M)}$. Then 
via the density of $H_{S_P}^1(\widetilde M)\cap L^2(\widetilde M)$ in $H_{S_P}^1(\widetilde M)$, we see that 
$ H_{S_P}^1(\widetilde M) \subset H_{max}^1(\widetilde M) $.

Then reverse direction $\|f\|_{H_{S_P}^1(\widetilde M)}\leq C \|f\|_{H_{max}^1(\widetilde M)}$ follows from Theorem \ref{main thm} and hence $ H_{max}^1(\widetilde M) \subset H_{S_P}^1(\widetilde M) $.

The proof is complete.
\end{proof}

Proposition \ref{prop max} provides the maximal function characterisation of $H^1(\widetilde M)$.

\begin{remark}
Theorem \ref{main thm}  and the above Proposition also give a new proof of the maximal function characterisation of product Hardy space associated with the Bessel operator in \cite{DLWY}.
\end{remark}

\bigskip
\bigskip
\noindent{\bf Acknowledgement:} Ji Li would like to thank Michael Cowling and Liangchuan Wu for suggestions.
Ji Li is supported by ARC DP 220100285.

\bigskip

\begin{flushleft}
School of Mathematical and Physical Sciences, Macquarie University, NSW 2109, Australia\\
E-mail address: ji.li@mq.edu.au
\end{flushleft}

\end{document}